\documentclass[12pt,reqno,tbtags]{amsart}
\usepackage[utf8]{inputenc}
\usepackage[T1]{fontenc}

\usepackage{amssymb}
\usepackage{amsthm}
\usepackage{amsmath, bm}
\usepackage{bbm}
\usepackage{mathrsfs}
\usepackage{txfonts}
\usepackage{graphicx}
\usepackage{geometry}
\usepackage{cite} 
\usepackage{hyperref}

\usepackage{float}

\usepackage{esint}
\usepackage{pgfplots}
\pgfplotsset{compat=1.18}
\usepackage{tikz-cd}
\usepackage{mdframed}
\usetikzlibrary{intersections}

\newtheorem{theorem}{Theorem}[section]
\newtheorem{proposition}{Proposition}[section]

\newtheorem{definition}{Definition}[section]    
\newtheorem{example}{Example}[section]
     
\newtheorem{lemma}{Lemma}[section]
\newtheorem{remark}{Remark}[section]

\allowdisplaybreaks[4]

\def\R{\mathbb{R}}
\def\C{\mathbb{C}}
\def\L{\mathcal{L}}
\def\Z{\mathbb{Z}}
\def\d{\mathrm{d}}
\def\N{\mathbb{N}}

\def\G{\mathcal{G}}

\def\E{\mathcal{E}}

\newcommand\norm[1]{\|#1\|}

\newcommand\1{\mathbbm{1}}

\newcommand\spt{\mathrm{spt}}

\def\d{\mathrm{d}}

\def\V{\mathcal{V}}
\newcommand{\ip}[2]{\left\langle #1,#2\right\rangle}

\begin{document}

\title[Observable sets for Schr\"odinger equations on combinatorial graphs]{Observable sets for Schr\"odinger equations on combinatorial graphs}

\author{Zhiqiang Wan}
\thanks{}
\address{School of Mathematical Sciences, University of Science and Technology of China, No. 96 Jinzhai Road, Baohe District, Hefei, Anhui Province, China}
\email{ZhiQiang\_Wan576@mail.ustc.edu.cn}

\author{Heng Zhang}
\thanks{}
\address{School of Mathematical Sciences, University of Science and Technology of China, No. 96 Jinzhai Road, Baohe District, Hefei, Anhui Province, China}
\email{hengz@mail.ustc.edu.cn}

\date{\today}

\begin{abstract}
We study observable sets for Schr\"odinger equations on combinatorial graphs.
For one-dimensional lattice Schr\"odinger operators
\(H=-\Delta_{\mathrm{disc}}+V\) with \(V(n)\to c\in\mathbb R\) as
\(|n|\to\infty\), we prove that a set \(E\subset\mathbb Z\) is observable at
some time, equivalently at any time, if and only if it satisfies a
local arithmetic condition. This reveals an arithmetic obstruction absent from
the Euclidean theory, where thickness is the decisive condition. The same
criterion also characterizes observability for the corresponding heat equation
on \(\mathbb Z\). In higher-dimensional lattices, we prove observability from
the complement of any finite set. We further obtain arithmetic criteria on discrete tori, showing that positive density alone does not
ensure observability.
\end{abstract}

\maketitle

\section{Introduction}
This work is situated within the framework of observability for
Schr\"odinger equations on different geometric spaces. In the present paper, we focus on
the discrete setting, where the underlying space is a combinatorial graph
\(G=(\mathcal V,\mathcal E)\). In this setting, the continuum Laplacian is
replaced by the discrete graph Laplacian \(\Delta_{\mathrm{disc}}\), defined for
\(f:\mathcal V\to\mathbb C\) by
\[
\Delta_{\mathrm{disc}} f(x)
:= \sum_{y\sim x} \bigl(f(y)-f(x)\bigr),
\]
where \(y\sim x\) means that \(x\) and \(y\) are adjacent vertices in \(G\).

Our motivation comes from the well-developed observability theory for
Schr\"odinger equations in continuous settings, such as \(\mathbb R^d\) and
\(\mathbb T^d\). In the graph setting, we study observable sets for
Schr\"odinger equations
\begin{equation}\label{eq:DS}
\begin{cases}
i\partial_t u(x,t)=Hu(x,t), & (x,t)\in \mathcal V\times(0,\infty),\\
u(x,0)=u_0(x)\in \ell^2(\mathcal V), & x\in \mathcal V,
\end{cases}
\end{equation}
where $H=-\Delta_{\mathrm{disc}}+V$ and \(V:\mathcal V\to\mathbb R\) is a real-valued potential. We use the following
notion of observability.

\begin{definition}
Let \(T>0\). A subset \(E\subset \mathcal V\) is called an
\textbf{observable set at time \(T\)} for \eqref{eq:DS} if there exists a
constant \(C_{\mathrm{obs}}=C_{\mathrm{obs}}(T,E)>0\) such that the observable inequality
\begin{equation}\label{eq:obs_ineq}
  \|u_0\|_{\ell^{2}(\mathcal V)}^{2}
  \le
  C_{\mathrm{obs}}
  \int_0^T\|\1_{E}u(\cdot,t)\|_{\ell^{2}(\mathcal V)}^{2}\,\mathrm{d}t
\end{equation}
holds for every solution \(u\) of \eqref{eq:DS} on \(G\). We say that \(E\) is an \textbf{observable set at some time} if it is observable
at time \(T\) for at least one \(T>0\). We say that \(E\) is an
\textbf{observable set at any time} if it is observable at time \(T\) for every
\(T>0\).
\end{definition}

Let us first recall the standard relation between observability and
controllability for Schr\"odinger dynamics. Given a subset \(E\subset\mathcal V\),
one may associate with \eqref{eq:DS} the internally controlled system
\[
\begin{cases}
 i\partial_t y(t)=Hy(t)+\1_E v(t), & t\in(0,T),\\
 y(0)=y_0\in \ell^2(\mathcal V),
\end{cases}
\]
where the control \(v\) acts only on \(E\). Since \(H\) is self-adjoint, the
homogeneous Schr\"odinger flow is unitary. By the Hilbert Uniqueness Method \cite{Li88a,Li88b,TW09}, exact
controllability of this system at time \(T\) is equivalent to the
observable inequality for the adjoint homogeneous equation,
\[
  \|u_0\|_{\ell^2(\mathcal V)}^2
  \le C_T \int_0^T \|\1_E e^{-itH}u_0\|_{\ell^2(\mathcal V)}^2\,\d t,
  \qquad u_0\in\ell^2(\mathcal V).
\]
Thus the problem of characterizing observable sets is, equivalently, the
problem of determining from which subsets one can control the Schr\"odinger
equation. In this paper we formulate our results in the language of
observability, but the above duality gives the corresponding controllability
interpretation.

To state our first main result on observable sets for the one-dimensional
lattice \(\mathbb Z\), we introduce the following definition.

\begin{definition}
For a finite set $S\subset\mathbb Z$, denote
\begin{equation}\label{def:local-gcd}
    \mathfrak g(S):=
    \begin{cases}
    \gcd\{|s-s'|:s,s'\in S,\ s\ne s'\}, & |S|\ge 2,\\[2mm]
    0, & |S|\le 1.
    \end{cases}
\end{equation}
Here \(\gcd\) denotes the greatest common divisor. We say a set $E\subset\mathbb Z$ satisfies \textbf{LAC} (local arithmetic condition) if there
exists $L\in\mathbb \N_+$ such that
\begin{equation}\label{eq:ULA-condition}
    \mathfrak g\big(E\cap[m-L,m+L]\big)=1,
    \qquad \forall m\in\mathbb Z .
\end{equation}
\end{definition}
By definition $\mathfrak g(S)=1$ means that the differences of points in $S$
generate the whole group $\mathbb Z$. With this terminology, we obtain the following characterization for \eqref{eq:DS} with asymptotically constant potentials.
\begin{theorem}\label{thm:1dZ}
Let $G=\Z$, $E\subset\mathbb{Z}$ and $V(n) \to c \in \R, \mathrm{as} \ |n| \to \infty$. Then the following are equivalent:
\begin{enumerate}
    \item [(i)]$E$ is observable at some time;
    \item [(ii)]$E$ is observable at any time;
    \item [(iii)]$E$ satisfies LAC.
\end{enumerate}
\end{theorem}

Theorem~\ref{thm:1dZ} reveals a qualitative distinction between
observability on the lattice \(\Z\) and observability on the Euclidean line
\(\R\). To make this comparison precise, we recall the standard notion of
thick sets.

\begin{definition}
 Let \(X\in\{\mathbb R^d,\mathbb Z^d\}\), and let \(\mu_X\) denote Lebesgue
measure when \(X=\mathbb R^d\) and counting measure when \(X=\mathbb Z^d\).
A \(\mu_X\)-measurable set \(E\subset X\) is called \textbf{\(\gamma\)-thick}, for
some \(\gamma\in(0,1]\), if there exists \(L>0\) \((L\in\mathbb N_+\) when
\(X=\mathbb Z^d)\) such that
\[
\mu_X\bigl(E\cap\{y\in X:\|y-x\|_\infty\le L\}\bigr)
\ge
\gamma\,\mu_X\bigl(\{y\in X:\|y-x\|_\infty\le L\}\bigr),
\qquad \forall x\in X.
\]   
\end{definition}

In the continuous one-dimensional setting \(X=\mathbb R\), Huang, Wang and Wang \cite[Theorem~1.1]{HWW22} prove that   thickness gives
the exact geometric characterization of observable sets: a measurable set
\(E\subset\mathbb R\) is observable at some time if and only if it is
\(\gamma\)-thick for some \(\gamma>0\). This follows from the
Logvinenko--Sereda uncertainty principle, which states that if $E \subset \R$ is a thick set, then for all $f$ in $L^2(\R)$ with Fourier support contained in a union of finitely many intervals, one has $\norm{f}_{L^{2}\left(\R\right)}\leq C \norm{f}_{L^{2}\left(E\right)}$, see, for example \cite{K01}. Moreover, Su, Sun and Yuan \cite[Theorem~1.2]{SSY25} extend
this characterization to every observability time \(T>0\) for general
one-dimensional Schr\"odinger operators with bounded continuous potentials.
Their argument is quantitative and provides an explicit upper bound for
the observability cost in terms of \(T\) and the thickness parameters.

By contrast, in the discrete setting Theorem~\ref{thm:1dZ} reveals an
arithmetic obstruction to observability. This obstruction is already visible in the borderline relation between
thickness and LAC, as the following remark illustrates.

\begin{remark}
The threshold
\(\gamma=\frac12\) is sharp: obviously every \(\frac12\)-thick set
\(E\subset\mathbb Z\) ensures LAC in
Theorem~\ref{thm:1dZ}, whereas the sublattice \(2\mathbb Z\) is
\(\gamma\)-thick for every \(\gamma<\frac12\) but is not observable for
any \(T>0\). Thus, below this threshold, thickness alone is not sufficient to determine
observability. The observability of a set is also influenced by its arithmetic structure. For example, \(9\mathbb Z\cup(9\mathbb Z+3)\) is not observable at any time, while
\(9\mathbb Z\cup(9\mathbb Z+1)\) is observable at any time by
Theorem~\ref{thm:1dZ}. These two sets have the same thickness
$\frac{2}{9}$, but they have different local arithmetic structures. We also remark that LAC implies that positive thickness is a necessary condition for observability.   
\end{remark}

We note that the threshold value $\frac{1}{2}$ for this phenomenon also appears in sampling theory, see, for instance, \cite[Lecture 6]{OU16}. Our approach to proving Theorem \ref{thm:1dZ} goes beyond the thickness-based argument from \cite{HWW22}. We use a
Hautus-type resolvent criterion together with the recurrence structure of
the one-dimensional discrete Schr\"odinger operator. This reduces the
problem to a local arithmetic obstruction: LAC rules out invisible
eigenfunctions, while failure of LAC produces explicit approximate
eigenfunctions vanishing on the observation set. Thus the characterization
depends not only on density, but on the local arithmetic structure of \(E\). We also note that the fact that \(\frac12\)-thickness implies observability for the free Schr\"odinger equation
at any time can also be obtained by a classical idea based on the
discrete Logvinenko--Sereda uncertainty principle; see Appendix \ref{app:LS-half-thick}.

A similar phenomenon also appears for the heat equation on the lattice.
For \(T>0\), we say that \(E\subset\mathbb Z\) is \textbf{heat observable}
at time \(T\) if there exists a constant \(C_T>0\) such that
\begin{equation}\label{eq:heat-observability-asymptotic-Z}
    \|u_0\|_{\ell^2(\mathbb Z)}^2
    \le
    C_T
    \int_0^T
    \|\1_E e^{-tH}u_0\|_{\ell^2(\mathbb Z)}^2\,\d t,
    \qquad
    \forall u_0\in\ell^2(\mathbb Z).
\end{equation}
Equivalently, since \(H\) is bounded and self-adjoint in the present
setting, one may replace \(\|u_0\|_{\ell^2}\) by
\(\|u(T)\|_{\ell^2}=\|e^{-TH}u_0\|_{\ell^2}\) at the cost of changing the
constant \(C_T\). In the continuous setting, observable sets for the heat equation on
\(\mathbb R^d\) are governed by thickness-type geometric conditions; see Wang, Wang, Zhang and Zhang \cite{WWZZ19}. Their proof proceeds through the chain ``thickness \(\Rightarrow\) spectral inequality \(\Rightarrow\) H\"older-type interpolation inequality \(\Rightarrow\) observability'', using a telescoping-series argument for the last implication and recovering the necessity of thickness by testing observability on a specially chosen heat solution. By contrast, for the discrete heat equation
on \(\mathbb Z\), Theorem~\ref{thm:heat-local-arithmetic-Z} shows that the
same local arithmetic condition as in Theorem~\ref{thm:1dZ} characterizes
both observability at some time and observability at any time. Under the same asymptotically constant assumption, the conclusion is as follows.

\begin{theorem}
\label{thm:heat-local-arithmetic-Z}
Let \(G=\mathbb Z\), let \(E\subset\mathbb Z\), and assume that
\(V(n)\to c\in\mathbb R\) as \(|n|\to\infty\). Then the following are
equivalent:
\begin{enumerate}
    \item[(i)] \(E\) is heat observable at some time for
    \(\partial_t u=-Hu\);
    \item[(ii)] \(E\) is heat observable at any time for
    \(\partial_t u=-Hu\);
    \item[(iii)] \(E\) satisfies LAC.
\end{enumerate}
\end{theorem}

Thus, for both the Schr\"odinger and heat dynamics on \(\mathbb Z\), the
relevant obstruction is not merely geometric density, but the local
arithmetic structure of the observable set.

Let us also mention two recent works on controllability and observability for
discrete heat equations. Bourroux, Jaming and Wang \cite{BJW26} study semi-discrete heat
equations with potentials on the full lattice \(h\mathbb Z^d\). They establish
spectral inequalities for the discrete Schr\"odinger operator
\(P_h=-\Delta_h+V\) on equidistributed sets, and then use the
Lebeau--Robbiano method together with the Hilbert Uniqueness Method to obtain
approximate null-controllability and relaxed observability estimates. Their
results cover bounded as well as polynomially growing potentials, and the
high-frequency remainder in their estimates reflects the failure of strong
unique continuation in the discrete setting. 

In another direction, M\"unch, Seifert, Stollmann and Tautenhahn \cite{MSST26} study heat
equations on weighted discrete graphs. They prove cost-uniform
\(\alpha\)-controllability from relatively dense sets by means of weak
observability estimates for the dual system. They also show that full
\(0\)-controllability cannot be expected in general. Our heat result complements and sharpens this picture
on the one-dimensional lattice. For asymptotically constant potentials,
the heat generator on \(\mathbb Z\) is a bounded self-adjoint operator, and
our observability estimate is therefore equivalent, by the usual
observability--controllability duality, to exact null-controllability of the
internally controlled heat equation. Hence exact null-controllability on
\(\mathbb Z\) is characterized by the same LAC as for Schr\"odinger observability.

For higher-dimensional lattices \(\mathbb Z^d\), \(d>1\), obtaining a
complete characterization analogous to LAC seems considerably more difficult.
In fact, the following example shows that high thickness does not imply
observability, in contrast to the one-dimensional result that
\(\frac{1}{2}\)-thickness implies observability.

\begin{example}
\label{rem:density-one-nonobservable-Z2}
$E:=\mathbb Z^2\setminus \{(m,n)\in\mathbb Z^2:m=n\}$ is \(\gamma\)-thick for every \(\gamma<1\). However, \(E\) is not
observable at any time for the free Schr\"odinger equation (that is, \(V=0\) in
\eqref{eq:DS}) ; see Proposition~\ref{prop:Z2}.
\end{example}

A comparable obstruction to a complete characterization is already present in
the Euclidean theory in higher dimensions. Indeed, while the necessity of
thickness for observable sets of the free Schr\"odinger equation on
\(\mathbb R^d\) follows from an argument which is insensitive to the dimension,
the corresponding sufficiency is much more delicate. As pointed out in
\cite[Remark~2.2(b)]{HWW22}, a resolvent-based proof of the implication
``thickness \(\Rightarrow\) observability'' in dimensions \(d\ge2\) would require
an uncertainty principle of the following form: if \(E\subset\mathbb R^d\) is
thick, then there exist constants \(C,\delta>0\) such that, for all
\(\lambda\ge1\),
\[
    \int_E |g(x)|^2\,\d x
    \ge C \int_{\mathbb R^d}|g(x)|^2\,\d x,
    \qquad
    \operatorname{spt}\mathcal Fg
    \subset
    \Bigl\{\xi\in\mathbb R^d:\bigl||\xi|-\lambda\bigr|
    \le \frac{\delta}{\lambda}\Bigr\}.
\]
This annular version of the Logvinenko--Sereda type uncertainty principle is not known in dimensions \(d\ge2\). Thus even in
the Euclidean setting, the exact characterization of observable sets for the
free Schr\"odinger equation in higher dimensions remains beyond the current
one-dimensional theory. A similar warning comes from sampling theory: after taking Fourier transform, lattice observability is closely related to sampling and uniqueness problems with prescribed frequency support, whose higher-dimensional behavior depends delicately on density, geometry and arithmetic structure; see, for instance,  \cite[Lecture~9]{OU16}.

Rather than pursuing such a
full characterization in higher dimensions, we focus on a robust class of
observable sets which is motivated by the corresponding Euclidean theory. For instance, Wang, Wang and Zhang
\cite[Remark~(a6)]{WWZ19} show that if \(E\) contains the exterior of a ball
\(B_r(0)\) for some \(r\ge0\), then \(E\) is observable at every time for the
free Schr\"odinger equation. In the lattice setting, an analogous exterior
observability phenomenon also holds: the complement of any finite set is
observable at any time. This gives the following result. 

\begin{theorem}\label{thm:Z^d}
Let \(d\ge1\), and let \(E\subset\mathbb Z^d\) be such that 
\(E^c:=\mathbb Z^d\setminus E\) is finite. Assume that \(V\in \ell^\infty(\mathbb Z^d;\mathbb R)\). Then \(E\) is observable at any
time \(T>0\).
\end{theorem}

We contrast the techniques employed in $\R^d$  with their discrete counterparts in $\Z^d$.  On $\mathbb{R}^d$, the corresponding exterior observability results \cite[Theorem~1.1 and Remark~(a6)]{WWZ19} rely crucially on quantitative Nazarov-type uncertainty principles in higher dimensions (see \cite{Jam07}). The Nazarov-type uncertainty principles state as follows. Given subsets $S,\,\Sigma\subset\R^{d}$ with $|S|, |\Sigma| < \infty$, there is a positive constant
$C(d, S, \Sigma) := C e^{C \min\{|S||\Sigma|, |S|^{1/d}w(\Sigma), |\Sigma|^{1/d}w(S)\}}$, (here  $w(\cdot)$ denotes mean width) with $C = C(d)$ such that for each $f \in L^2(\R^{d};\C)$,
\begin{equation*}
    \int_{\R_{x}^{d}} |f(x)|^2 \,\d x \le C(d, S, \Sigma) \left( \int_{\R_{x}^{d}\setminus S} |f(x)|^2 \,\d x + \int_{\R_{\xi}^{d}\setminus \Sigma} |\mathcal Ff(\xi)|^2 \,\d\xi \right).
\end{equation*}
In our lattice setting, we avoid such microlocal machinery: the proof is based on a $TT^{*}$ operator-norm argument together with specific spectral features of the discrete Laplacian. 

It is classical that on a compact Riemannian manifold, any open set satisfying
GCC
(geometric control condition) is observable at any time for the wave and
Schr\"odinger equations (see e.g. \cite{L92, P01}). Furthermore, GCC is known to be necessary on manifolds with
periodic geodesic flows (Zoll manifolds), see \cite{M11}. Related propagation,
observation and control problems have also been studied on one-dimensional
continuous networks, or metric-graph type models. In particular, D\'ager and
Zuazua \cite{DZ06} gave a systematic treatment of wave propagation,
observation and control for networks of flexible strings distributed along
planar graphs; their analysis also illustrates the role of graph topology,
edge lengths and Diophantine phenomena in observability questions. A more
recent formulation of GCC ideas on metric graphs can be found,
for instance, in \cite{ADJL25}. For finite combinatorial graphs, however,
there are no continuous geodesic flows nor microlocal ray-propagation
mechanisms; accordingly, GCC plays no role in the discrete combinatorial
setting. This shifts the problem toward
spectral and arithmetic properties of the graph Laplacian. We illustrate this
phenomenon on discrete tori.

The discrete \(d\)-dimensional tori $T_N^d:=(\Z/N\Z)^d$, $N\geq 3$ are endowed with the nearest-neighbor graph structure: $x,y\in (\Z/N\Z)^d$ are adjacent iff $y=x\pm e_j$ for some standard basis vector $e_j$, with arithmetic taken modulo $N$. On these finite graphs, we restrict ourselves to the free Schr\"odinger
equation, in order to isolate the arithmetic effects coming from the
spectrum of the discrete Laplacian.

We give an arithmetic characterization of observable sets on \(T_N^1\).
We also construct, in every dimension, positive-density subsets of \(T_N^d\)
which are not observable at any time.

\begin{theorem}\label{thm:tori}
Consider the free Schr\"odinger equation on the discrete torus \(T_N^d\). Then the following hold.
\begin{enumerate}
    \item[(i)] In dimension \(d=1\), a non-empty set
    \(E\subset \mathbb Z/N\mathbb Z\) is observable at any time if and only if
    \[
        \gcd\left\{
            \frac{N}{\gcd(N,2)},\ x-y:\ x,y\in E
        \right\}=1 .
    \]

    \item[(ii)] For every \(d\ge 1\), there exist \(c>0\) and an infinite
    sequence \(N_m\to\infty\) such that, for each \(m\), there is a set
    \(E_{N_m}\subset(\mathbb Z/N_m\mathbb Z)^d\) satisfying
    \[
        |E_{N_m}|\ge c\,N_m^d,
    \]
    but \(E_{N_m}\) is not observable at any time.
\end{enumerate}
\end{theorem}

 We remark that on continuous tori $\mathbb{T}^d$, every non-empty open set is observable at any time (see, e.g. \cite{AM14, BBZ13, BZ12, Har89, Jaf90, Kom92, Tao21}). Furthermore, every positive-measure set on $\mathbb{T}^2$ is observable at any time \cite{BZ19}, and a key ingredient in the proof is the construction of semiclassical defect measures to enforce mass concentration along rational directions. In contrast, Theorem \ref{thm:tori} shows that due to arithmetic structure effects, some discrete tori contain high-density subsets that remain not observable at any time as the number of vertices tends to infinity.

Finally, beyond flat Euclidean and toroidal geometries, another important
source of observability phenomena comes from negatively curved spaces. We mention that there are also studies of observability for Schr\"odinger equations on negatively curved manifolds, see for instance \cite{AnRi12, DJ18, DJN22, Jin18}. We also highlight the quantitative Hautus-type resolvent estimates established in \cite{DJ18, DJN22}. A central idea in those works is the fractal uncertainty principle, introduced in \cite{DZ16} and recently extended to higher dimensions in \cite{Coh25}. As a discrete counterpart of negatively curved spaces, regular trees also provide
a natural testing ground for observability phenomena. It would be interesting
to obtain an intrinsic characterization of observable sets on tree graphs,
where the exponential volume growth and boundary structure may play a role
analogous to the hyperbolic dynamics in the manifold setting.

\begin{figure}[htbp]
\centering
\resizebox{0.92\textwidth}{!}{%
\begin{tikzpicture}[
    x=1cm,y=1cm,
    >=Stealth,
    box/.style={
        draw,
        rectangle,
        line width=0.4pt,
        align=center,
        inner sep=4pt,
        minimum height=8mm,
        text width=#1,
        font=\scriptsize\itshape
    },
    box/.default=32mm,
    arrow/.style={->, line width=0.4pt},
    lab/.style={
        font=\scriptsize\itshape,
        fill=white,
        inner sep=1pt
    }
]

\node[box=50mm] (B) at (0,-1.2)
{observable sets for \\discrete Schr\"odinger equation};

\node[box=15mm] (C) at (-5.4,-3.1)
{$\mathbb Z$};

\node[box=15mm] (T) at (0,-3.1)
{$T_N^d$};

\node[box=24mm] (J) at (5.4,-3.1)
{$\mathbb Z^d,\ d>1$};

\node[box=15mm] (D) at (-5.4,-4.9)
{LAC};

\node[box=38mm] (E) at (-5.4,-6.7)
{Hautus criterion\\
recurrence structure};

\node[box=25mm] (F) at (-7.6,-8.9)
{observable at\\
some time $=$ any time};

\node[box=32mm] (G) at (-3.2,-8.9)
{heat observability\\
same LAC criterion};

\node[box=38mm] (U) at (0,-5.3)
{positive density does\\
not ensure observability};

\node[box=30mm] (K1) at (3.2,-7.0)
{Example \ref{rem:density-one-nonobservable-Z2}};

\node[box=36mm] (K2) at (7.6,-7.0)
{exterior observability};

\draw[arrow] (B) -- (C);
\draw[arrow] (B) -- (T);
\draw[arrow] (B) -- (J);
\draw[arrow] (C) -- (D);
\draw[arrow] (D) -- (E);
\draw[arrow] (E) -- (D);
\draw[arrow] (F) -- (E);
\draw[arrow] (E) -- (G);
\draw[arrow] (T) -- (U);
\draw[arrow] (J) -- node[lab,above left] {} (K1);
\draw[arrow] (J) -- node[lab,above right] {} (K2);

\end{tikzpicture}%
}
\caption{Logical structure of the main results.}
\label{fig:main-structure}
\end{figure}

The rest of the paper is organized as follows. In Section~\ref{sec:2}, we prove the
characterization of observable sets on \(\mathbb Z\) for both Schr\"odinger
and heat equations, and establish exterior observability on higher-dimensional
lattices. In Section~\ref{sec:3}, we study finite graphs and discrete tori, proving the
arithmetic characterization in one dimension and constructing positive-density
non-observable sets in higher dimensions. Appendix \ref{app:LS-half-thick} gives an alternative
proof of the half-thickness result on \(\mathbb Z\) via a discrete
Logvinenko--Sereda argument. Figure \ref{fig:main-structure} exhibits the logical structure of the main results.

\section{Observable sets on lattices}\label{sec:2}

\subsection{Proof of Theorem~\ref{thm:1dZ}}
\label{sec:proof-thm-1dZ}

Below we set
\[
    W(n):=V(n)-c,\quad H_0:=-\Delta_{\rm disc}+c .
\]
Since \(W(n)\to 0\) as \(|n|\to \infty\), the multiplication operator by \(W\) is compact on \(\ell^2(\Z)\). Indeed, if we denote
\[
    W_R(n):=W(n)\cdot\1_{\{|n|\le R\}},
\]
then multiplication by \(W_R\) has finite rank and
\[
    \|W-W_R\|_{\mathcal L(\ell^2(\mathbb Z))}=\sup_{|n|>R}|W(n)|\to 0\quad{\rm as}\;R\to \infty.
\]

We first prove that observability at some time is equivalent to LAC. We begin with the following uniform local estimate.
\begin{lemma}\label{lem:local-estimate-H0-c}
Fix \(L\in\N_+\) and \(K>0\). Then there exists a constant
\(C_{L,K,c}>0\) such that for every
\[
    S\subset[-L,L]\cap\Z\quad{\rm with}\quad\mathfrak g(S)=1,
\]
every \(\lambda\in[-K,K]\), and every vector
\[
    v=(v_n)_{n=-L-1}^{L+1}\in\C^{2L+3},
\]
one has
\begin{equation}\label{eq:local-estimate-H0-c}
    \sum_{n=-L}^{L}|v_n|^2\le C_{L,K,c}\left[\sum_{n=-L}^{L}|-v_{n+1}-v_{n-1}+(2+c-\lambda)v_n|^2+\sum_{s\in S}|v_s|^2\right].
\end{equation}
\end{lemma}

\begin{proof}
Suppose that the estimate fails. Then there exist \(S_j\subset[-L,L]\cap\Z\) with \(\mathfrak g(S_j)=1\), \(\lambda_j\in[-K,K]\), and \(v^{(j)}=(v_n^{(j)})_{n=-L-1}^{L+1}\in\C^{2L+3}\) such that
\[
    \sum_{n=-L}^{L}|v_n^{(j)}|^2=1,
\]
whereas
\[
    \sum_{n=-L}^{L}|-v_{n+1}^{(j)}-v_{n-1}^{(j)}+(2+c-\lambda_j)v_n^{(j)}|^2+\sum_{s\in S_j}|v_s^{(j)}|^2\to 0\quad{\rm as}\;j\to\infty.
\]
Since there are only finitely many subsets of \([-L,L]\cap\Z\), and since \([-K,K]\) is compact, after passing to a subsequence we may assume that
\[
    S_j=S,\quad\lambda_j\to\lambda_\infty\in[-K,K].
\]
The coordinates \(v_n^{(j)}\), \(-L-1\le n\le L+1\), are uniformly bounded. Thus, after passing to a further subsequence, we deduce \(v^{(j)}\to v\) for some \(v\in \C^{2L+3}\). The limit satisfies
\begin{gather}
   \sum_{n=-L}^{L}|v_n|^2=1, \label{eq:finitesumvn}
   \end{gather}
and \begin{gather}
   -v_{n+1}-v_{n-1}+(2+c-\lambda_\infty)v_n=0
   \quad (-L\le n\le L),
   \qquad
   v_s=0\quad (s\in S). \notag
\end{gather}
For simplicity we denote \(a:=2+c-\lambda_\infty\). 
The recurrence extends uniquely to a solution on the whole lattice. We show that such a solution must be zero if it vanishes on \(S\).

If \(|a|<2\), then \(a=2\cos\theta\) for some \(\theta\in(0,\pi)\), and
\[
    v_n=Ae^{in\theta}+Be^{-in\theta}.
\]
Suppose the solution is nonzero and vanishes on \(S\). Then we derive \(A,B\ne 0\), and for every \(s\in S\),
\[
    Ae^{is\theta}+Be^{-is\theta}=0.
\]
Hence for all \(s,s'\in S\),
\[
    e^{2i(s-s')\theta}=1.
\]
Since we have assumed \(\mathfrak g(S)=1\), the differences \(s-s'\) generate \(\Z\), and therefore \(e^{2i\theta}=1\). This contradicts \(\theta\in(0,\pi)\).

When \(a\in\{\pm 2\}\), we get that
\[
    v_{n}=({\rm sgn}\,a)^{n}\cdot(A+Bn).
\]
Using the assumption \(\mathfrak g(S)=1\) we note that the set \(S\) contains at least two distinct integers, and this implies that \(A=B=0\).

Finally, if \(|a|>2\), then there exist \(r>1\) and \(\sigma\in\{\pm 1\}\) such that
\[
    a=\sigma(r+r^{-1}),
\]
and every solution has the form
\[
    v_n=\sigma^n(Ar^n+Br^{-n}).
\]
If this solution vanishes at two distinct points \(s_1,s_2\in S\), then a nonzero solution would imply
\[
    r^{2(s_1-s_2)}=1,
\]
which is impossible because \(r>1\) and \(s_1\ne s_2\). Thus again \(A=B=0\).

Therefore the only solution of the recurrence which vanishes on \(S\) is the zero solution. This contradicts \eqref{eq:finitesumvn}. The lemma follows.
\end{proof}

Before proving the resolvent estimate, we exclude a possible spectral obstruction. The following lemma says that, under LAC, no nonzero eigenfunction of \(H\) can be completely invisible from the observation set \(E\). Set
\[
    \mathcal I_E:=\{\psi\in\ell^2(\Z): \psi(n)=0\;{\rm for\;every}\;n\in E\}.
\]
Thus \(\mathcal I_E\) is the closed subspace of \(\ell^2(\Z)\) consisting of functions which vanish on the observation set.

\begin{lemma}\label{lem:no-hidden-eigenfunctions-asymptotic-c}
Assume that \(E\subset\mathbb Z\) satisfies LAC. Then
\begin{equation}\label{eq:no-hidden-eigenfunctions-asymptotic-c}
    \ker(H-\lambda)\cap \mathcal I_E=\{0\},\quad \forall\lambda\in\R.
\end{equation}
\end{lemma}

\begin{proof}
Let \(L\in\N_+\) be such that
\[
    \mathfrak g(E\cap[m-L,m+L])=1\quad{\rm for\;all}\;m\in\Z.
\]
Fix a choice \(\psi\in\ker(H-\lambda)\cap \mathcal I_E\). By the definition of \(\mathcal I_E\), we have \(\psi|_E=0\). If \(\psi\ne0\), then necessarily $|\lambda|\le \|H\|_{\mathcal L(\ell^{2}(\Z))}$. Choose \(K:=\|H\|_{\mathcal L(\ell^{2}(\Z))}+1\). Let \(C_{L,K,c}\) be the constant in
Lemma~\ref{lem:local-estimate-H0-c}. Since \(W(n)\to 0\) as \(|n|\to \infty\), there exists
\(R>0\) such that
\begin{equation}\label{eq:W-small-tail}
    |W(n)|^2\le\frac{1}{2C_{L,K,c}}\quad{\rm when}\;|n|\ge R.
\end{equation}
Choose 
\[
    m\in\left\{i:[i-L,i+L]\subset\{n\in\Z: |n|\ge R\}\right\}.
\]
Since \(\mathfrak g((E-m)\cap[-L,L])=1\), applying Lemma~\ref{lem:local-estimate-H0-c} to the sequence \((\psi(m+\cdot))_{-L-1}^{L+1}\in\C^{2L+3}\) while using the facts
\[
    (H_0-\lambda)\psi=-W\psi\quad{\rm and}\quad \psi(m+s)=0\quad {\rm for\;any}\;s\in (E-m)\cap[-L,L],
\]
we deduce that
\[
    \sum_{n=-L}^{L}|\psi(m+n)|^2\le C_{L,K,c}\sum_{n=-L}^{L}|W(m+n)|^2|\psi(m+n)|^2.
\]
Using \eqref{eq:W-small-tail}, we get
\[
    \sum_{n=-L}^{L}|\psi(m+n)|^2\le\frac{1}{2}\sum_{n=-L}^{L}|\psi(m+n)|^2.
\]
Thus we have \(\psi(n)\equiv 0\) for \(m-L\le n\le m+L\). Since this argument holds for every \(m\in\Z\) whose \(L\)-neighborhood lies in \(\{|n|\ge R\}\), we show \(\psi\) has finite support.

In fact, there is no nonzero finitely supported solution \(\psi\) satisfying \(H\psi=\lambda\psi\). To see this, if \(\psi\ne 0\), let \(N\) be the largest integer such that \(\psi(N)\ne0\). Evaluating \(H\psi=\lambda\psi\) at \(N+1\), we get
\[
    -\psi(N+2)-\psi(N)+(2+V(N+1))\psi(N+1)=\lambda\psi(N+1).
\]
All terms vanish except \(-\psi(N)\). Hence \(\psi(N)=0\), a contradiction. Therefore \(\psi=0\).
\end{proof}

Conversely, when LAC fails, the arithmetic defect can be used to construct
states which are invisible on \(E\) and which almost solve the free equation
generated by \(H_0\). The additional \(\ell^\infty\)-smallness will later
allow us to pass from \(H_0\) to the compactly perturbed operator
\(H=H_0+W\).
\begin{lemma}[Arithmetic quasimodes when LAC fails]
\label{lem:arithmetic-quasimodes-H0}
If \(E\) does not satisfy LAC, then there exist \(f_N\in\ell^2(\Z)\) and \(\mu_N\in[c,c+4]\) such that
\begin{equation}\label{eq:arithmetic-quasimodes-H0}
    \|f_N\|_{\ell^2(\Z)}=1,\quad f_N|_E=0,\quad{\rm and}\quad\|(H_0-\mu_N)f_N\|_{\ell^2(\Z)}\to 0,\;\|f_N\|_{\ell^\infty(\Z)}\to 0\quad{\rm as}\;N\to\infty.
\end{equation}
\end{lemma}

\begin{proof}
Since LAC fails, for every \(N\in\mathbb N\) there exists
\(m_N\in\mathbb Z\) such that
\[
    \mathfrak g(E\cap I_N)\ne 1\quad{\rm where}\;I_N:=[m_N-N,m_N+N]\cap\mathbb Z.
\]
We distinguish three cases.

\medskip

\noindent
\textbf{Case 1: \(E\cap I_N=\emptyset\).}
Taking \(f_{N}=\1_{I_{N}}\) and \(\mu_{N}=c\), we have \(f_N|_E=0\) and \(\|f_{N}\|_{\ell^{2}(\Z)}^{2}=2N+1\). Since \((H_0-\mu_N)f_N=-\Delta_{\rm disc}f_N\) is supported near the two endpoints of \(I_N\), we have
\[
    \|(H_0-\mu_N)f_N\|_{\ell^2(\Z)}\le O(1).
\]

\medskip

\noindent
\textbf{Case 2: \(E\cap I_N=\{r_N\}\).}
Take \(f_N(n)=(n-r_N)\cdot\1_{I_N}(n)\) and \(\mu_N=c\), then \(f_N|_E=0\), and, moreover,
\[
    \|f_N\|_{\ell^2(\Z)}^2=\sum_{n=m_N-N}^{m_N+N}|n-r_N|^2\ge c_0 N^3
\]
for some absolute constant \(c_0>0\). Since \(n\mapsto n-r_N\) is harmonic for \(\Delta_{\rm disc}\), the residual
\[
    (H_0-\mu_N)f_N=-\Delta_{\rm disc}f_N
\]
is supported near the endpoints and has size \(O(N)\). Thus \(\|(H_0-\mu_N)f_N\|_{\ell^2(\Z)}\le C N\) and \(\|f_N\|_{\ell^\infty(\Z)}\le C N\) for some absolute constant \(C>0\).

\medskip

\noindent
\textbf{Case 3: \(|E\cap I_N|\ge2\).}
In this case, we have \(2\le\mathfrak g(E\cap I_N)\le 2N\). Choose \(r_N\in E\cap I_N\). Then
\[
    E\cap I_N\subset r_N+\mathfrak g(E\cap I_N)\Z.
\]
Below we denote
\[
    \theta_N:=\frac{\pi}{\mathfrak g(E\cap I_N)}\quad{\rm and}\quad\mu_N:=c+2-2\cos\theta_N,
\]
and define
\[
    f_N(n):=\1_{I_N}(n)\sin\big(\theta_N(n-r_N)\big).
\]
Then \(\mu_N\in[c,c+4]\) and \(f_N|_E=0\). Note that 
\[
    (H_0-\mu_N)\sin\big(\theta_N(n-r_N)\big)=0\quad{\rm on}\;\Z,
\]
and then \((H_0-\mu_N)f_N\) is supported only near the two endpoints of
\(I_N\). Moreover, we derive 
\[
    \|(H_0-\mu_N)f_N\|_{\ell^2}\le O(1).
\]
To estimate the lower bound of \(\|f_N\|_{\ell^2}\), we consider the sequence
\[
    k\mapsto \sin^2\left(\frac{\pi k}{\mathfrak g(E\cap I_N)}\right),
\]
which is \(\mathfrak g(E\cap I_N)\)-periodic, and satisfies
\[
    \sum_{k=0}^{\mathfrak g(E\cap I_N)-1}\sin^2\left(\frac{\pi k}{\mathfrak g(E\cap I_N)}\right)=\frac{\mathfrak g(E\cap I_N)}{2}.
\]
Every interval of \(2N+1\) consecutive integers contains
\(\left\lfloor\frac{2N+1}{\mathfrak g(E\cap I_N)}\right\rfloor\) disjoint complete periods. Thus
\[
    \|f_N\|_{\ell^2(\Z)}^2\ge\left\lfloor\frac{2N+1}{\mathfrak g(E\cap I_N)}\right\rfloor\cdot\frac{\mathfrak g(E\cap I_N)}{2}.
\]
Since \(\mathfrak g(E\cap I_N)\le2N\), we have
\[
    \left\lfloor\frac{2N+1}{\mathfrak g(E\cap I_N)}\right\rfloor\cdot\frac{\mathfrak g(E\cap I_N)}{2}\geq\frac{2N+1}{4}\implies\|f_N\|_{\ell^2(\Z)}^2\ge\frac{2N+1}{4}.
\]
A trivial estimate shows that
\[
    \|f_N\|_{\ell^\infty(\Z)}\le 1.
\]

\medskip

In all three cases define
\[
    g_N:=\frac{f_N}{\|f_N\|_{\ell^2(\Z)}}.
\]
Then
\[
    \|g_N\|_{\ell^2(\Z)}=1\quad{\rm and}\quad g_N|_E=0,
\]
and the above estimates imply
\[
    \max\left\{\|(H_0-\mu_N)g_N\|_{\ell^2(\Z)},\|g_N\|_{\ell^\infty(\Z)}\right\}\leq\frac{C}{\sqrt N}\to 0\quad{\rm as}\;N\to\infty.
\]
The lemma is proved.
\end{proof}

Recall that, in the functional framework from Miller~\cite{M05}, he considered
the observability of the following system:
\begin{equation}\label{eq:Ob system}
    \dot{x}(t)-i\mathscr{A}x(t)=0,\qquad
    x(0)=x_{0}\in X,\qquad
    y(t)=\mathscr{C}x(t) \in Y,
\end{equation}
where \(X\) and \(Y\) are Hilbert spaces,
\(\mathscr{A}:D(\mathscr{A})\subset X\to X\) is a self-adjoint operator, and
\(\mathscr C\in\mathcal L(D(\mathscr A),Y)\) is an admissible observation
operator.  Miller proved that
\begin{theorem}\cite[Theorem 5.1]{M05}\label{thm:M05}
The system \eqref{eq:Ob system} is exactly observable, i.e., the observable inequality holds at some $T>0$:
\begin{equation}\label{ineq:Ob exactly ob}
    \forall x_{0}\in X,\;\norm{x_{0}}^{2}_{X}\leq \kappa_{T}\int_{\left(0,T\right)}\norm{y}^{2}_{Y}\d t,
\end{equation}
if and only if,  there exist $m, M>0$ such that the following {\it observable resolvent estimate} holds:
\begin{equation}\label{ineq:Ob resolvent estimate}
    \forall\;x\in D\left(\mathscr{A}\right),\forall\lambda\in\R,\;\;\norm{x}^{2}_{X}\leq M\norm{\left(\mathscr{A}-\lambda\right)x}^{2}_{X}+m\norm{\mathscr{C}x}^{2}_{Y}.
\end{equation}
Moreover, for all $\varepsilon>0$, there is a $C_{\varepsilon}>0$ such that \eqref{ineq:Ob resolvent estimate} implies \eqref{ineq:Ob exactly ob} for all $T>\sqrt{M\left(\pi^{2}+\varepsilon\right)}$ with $\kappa_{T}=C_{\varepsilon}mT/(T^{2}-M\left(\pi^{2}+\varepsilon\right))$. 
\end{theorem}

In our one-dimensional lattice setting, we apply this theorem with
\begin{align}\label{eq:setting}
 X=Y=\ell^2(\Z),\quad\mathscr A=-H\quad{\rm and}\quad\mathscr C=\1_E\cdot{\rm Id}_{\ell^2(\Z)},
\end{align}

Now we use Theorem \ref{thm:M05} to prove the following proposition.
\begin{proposition}
\label{prop:some-time-LAC-asymptotic-c}
Let $G=\Z$, $E\subset\mathbb{Z}$ and $V(n) \to c \in \R, \mathrm{as} \ |n| \to \infty$. Then the following are equivalent:
\begin{enumerate}
    \item[(i)] \(E\) is observable at some time for \eqref{eq:DS};
    \item[(ii)] \(E\) satisfies LAC.
\end{enumerate}
\end{proposition}

\begin{proof}
By Theorem~\ref{thm:M05}, applied with \eqref{eq:setting}
observability at some time is equivalent to the existence of constants
\(M,m>0\) such that
\begin{equation}\label{eq:H-Hautus-asymptotic-c}
    \|f\|_{\ell^2(\Z)}^2\le M\|(H-\lambda)f\|_{\ell^2(\Z)}^2+m\|f\|_{\ell^2(E)}^2
\end{equation}
for all \(\lambda\in\R\) and all \(f\in \ell^2(\Z)\).

We first prove \((i)\Rightarrow(ii)\). Assume that
\eqref{eq:H-Hautus-asymptotic-c} holds. Suppose, by contradiction, that
LAC fails. By Lemma~\ref{lem:arithmetic-quasimodes-H0}, there exist
\(f_N\in\ell^2(\Z)\) and \(\mu_N\in[c,c+4]\) such that
\[
    \|f_N\|_{\ell^2(\Z)}=1,\quad f_N|_E=0,\quad{\rm and}\quad\|(H_0-\mu_N)f_N\|_{\ell^2(\Z)}\to 0,\;\|f_N\|_{\ell^\infty(\Z)}\to 0\quad{\rm as}\;N\to\infty.
\]
We claim that
\[
    \|Wf_N\|_{\ell^2(\Z)}\to 0.
\]
Indeed, fix \(\varepsilon>0\). Choose \(R>0\) such that
\[
    \sup_{|n|>R}|W(n)|\le\varepsilon.
\]
Then
\[
\begin{aligned}
    \|Wf_N\|_{\ell^2(\Z)}^2&\le\|W\|_{\ell^\infty(\Z)}^2\sum_{|n|\le R}|f_N(n)|^2+\varepsilon^2\sum_{|n|>R}|f_N(n)|^2\\
    &\le\|W\|_{\ell^\infty(\Z)}^2(2R+1)\|f_N\|_{\ell^\infty(\Z)}^2+\varepsilon^2.
\end{aligned}
\]
Letting \(N\to\infty\) gives
\[
    \limsup_{N\to\infty}\|Wf_N\|_{\ell^2(\Z)}^2\le\varepsilon^2.
\]
Since \(\varepsilon>0\) is arbitrary,
\[
    \|Wf_N\|_{\ell^2(\Z)}\to0.
\]
Consequently,
\[
    \|(H-\mu_N)f_N\|_{\ell^2(\Z)}\le\|(H_0-\mu_N)f_N\|_{\ell^2(\Z)}+\|Wf_N\|_{\ell^2(\Z)}\to 0.
\]
Since \(f_N|_E=0\), applying \eqref{eq:H-Hautus-asymptotic-c} to the pairs \((f_N,\mu_N)\) gives
\[
    1=\|f_N\|_{\ell^2(\Z)}^2\le M\|(H-\mu_N)f_N\|_{\ell^2(\Z)}^2\to 0,
\]
a contradiction. Hence LAC holds.

We now prove \((ii)\Rightarrow(i)\). Assume that \(E\) satisfies LAC.
We prove the Hautus estimate \eqref{eq:H-Hautus-asymptotic-c}. Suppose
that it fails. Then there exist \(f_j\in\ell^2(\Z)\) and
\(\lambda_j\in\R\) such that
\begin{equation}\label{eq:bad-sequence-H-asymptotic-c}
    \|f_j\|_{\ell^2(\Z)}=1,\quad{\rm and}\quad\|(H-\lambda_j)f_j\|_{\ell^2(\Z)}\to0,
    \;\|f_j\|_{\ell^2(E)}\to 0\quad{\rm as}\;j\to\infty.
\end{equation}
The sequence \((\lambda_j)\) is bounded. Indeed,
\[
    |\lambda_j|=\|\lambda_j f_j\|_{\ell^2(\Z)}\le\|Hf_j\|_{\ell^2(\Z)}+\|(H-\lambda_j)f_j\|_{\ell^2(\Z)}\le\|H\|_{\mathcal L(\ell^2(\Z))}+\|(H-\lambda_j)f_j\|_{\ell^2(\Z)}.
\]
After passing to a subsequence, we may assume that
\[
    \lambda_j\to\lambda\in\R,\quad f_j\rightharpoonup f\quad{\rm weakly\;in}\;\ell^2(\Z).
\]
In fact, \(f=0\) and we will prove it by showing a contradiction. If \(f\ne0\), then
\[
    (H-\lambda)f_j=(H-\lambda_j)f_j+(\lambda_j-\lambda)f_j\to 0 \quad{\rm strongly\;in}\;\ell^2(\Z).
\]
Passing to the weak limit gives
\[
    (H-\lambda)f=0.
\]
Moreover, \(\|f_j\|_{\ell^2(E)}\to0\) implies \(f|_E=0\). Hence
\[
    0\ne f\in\ker(H-\lambda)\cap\mathcal I_E,
\]
contradicting Lemma~\ref{lem:no-hidden-eigenfunctions-asymptotic-c}.
Therefore \(f=0\).

Since multiplication by \(W\) is compact and \(f_j\rightharpoonup0\), we
have \(\|Wf_j\|_{\ell^2(\Z)}\to 0\). Write
\[
    (H_0-\lambda_j)f_j=(H-\lambda_j)f_j-Wf_j.
\]
Combining \eqref{eq:bad-sequence-H-asymptotic-c}, we deduce
\begin{equation}\label{eq:H0-residual-small-asymptotic-c}
    \|(H_0-\lambda_j)f_j\|_{\ell^2(\Z)}\to 0.
\end{equation}

Choose \(K>\|H\|_{\L(\ell^2(\Z))}+1\). After discarding finitely many terms, we may assume that \(\lambda_{j}\in[-K,K]\) for any \(j\). Let \(L\in\N_+\) be the LAC length for \(E\), and let
\(C_{L,K,c}\) be the constant from
Lemma~\ref{lem:local-estimate-H0-c}. For each \(m\in\mathbb Z\), we have \(\mathfrak g((E-m)\cap[-L,L])=1\). Applying Lemma~\ref{lem:local-estimate-H0-c} to the sequence \((f_{j}(m+\cdot))_{-L-1}^{L+1}\in\C^{2L+3}\), one has
\[
    \sum_{n=m-L}^{m+L}|f_j(n)|^2\le C_{L,K,c}\left[\sum_{n=m-L}^{m+L}|(H_0-\lambda_j)f_j(n)|^2+\sum_{s\in E\cap[m-L,m+L]}|f_j(s)|^2\right].
\]
Summing over all \(m\in\Z\), and using that every lattice point is
counted exactly \(2L+1\) times, we obtain
\[
    \|f_j\|_{\ell^2(\Z)}^2\le C_{L,K,c}\|(H_0-\lambda_j)f_j\|_{\ell^2(\Z)}^2+C_{L,K,c}
    \|f_j\|_{\ell^2(E)}^2.
\]
By \eqref{eq:H0-residual-small-asymptotic-c} and
\eqref{eq:bad-sequence-H-asymptotic-c}, the right-hand side tends to
\(0\). This contradicts \(\|f_j\|_{\ell^2(\Z)}=1\) for any \(j\). Thus \eqref{eq:H-Hautus-asymptotic-c} holds. By Theorem~\ref{thm:M05},
\(E\) is observable at some time.
\end{proof}

Here and below, \(\mathcal L(X,Y)\) denotes the space of bounded linear
operators from \(X\) to \(Y\), and we write \(\mathcal L(X):=\mathcal L(X,X)\).

\begin{lemma}
\label{lem:some-time-any-time-bounded-asymptotic-c}
Let \(X,Y\) be Hilbert spaces. Let \(\mathscr A\in\mathcal L(X)\) be self-adjoint
and let \(\mathscr C\in\mathcal L(X,Y)\). Set
\[
    U(t):=e^{-it\mathscr A},\quad t\in\R.
\]
Assume that there exist \(T_0>0\) and \(c_0>0\) such that
\[
    \int_0^{T_0}\|\mathscr CU(t)x\|_Y^2\,\d t\ge c_0\|x\|_X^2,\quad{\rm for\;any}\;x\in X.
\]
Then for every \(T>0\) there exists \(c_T>0\) such that
\[
    \int_0^T\|\mathscr CU(t)x\|_Y^2\,\d t\ge c_T\|x\|_X^2,\quad{\rm for\;any}\;x\in X.
\]
\end{lemma}

\begin{proof}
If \(T\ge T_0\), the conclusion is immediate. We assume
\(0<T<T_0\) and argue by contradiction. Suppose that the conclusion fails
for this \(T\). Then there exists \(x_j\in X\) such that
\[
    \|x_j\|_X=1\quad{\rm and}\quad\int_0^T\|\mathscr CU(t)x_j\|_Y^2\,\d t\to0.
\]
For simplicity we denote 
\[
    Q:=\mathscr C^*\mathscr C\quad{\rm and}\quad Q(t):=U(-t)QU(t),
\]
and define
\[
    q_j(t):=\langle Q(t)x_j,x_j\rangle_X=\|\mathscr CU(t)x_j\|_Y^2,\quad 0\le t\le T_0.
\]
Since \(\mathscr A\) is bounded, \(t\mapsto Q(t)\) is \(C^\infty\) in operator norm. For every \(k\ge0\), we calculate that 
\[
    \frac{\d^k}{\d t^k}Q(t)=i^k U(-t)\,{\rm ad}_{\mathscr A}^k(Q)\,U(t)\quad{\rm where}\quad{\rm ad}_{\mathscr A}(Q):=\mathscr AQ-Q\mathscr A.
\]
Note that \(\|{\rm ad}_{\mathscr A}(Q)\|\leq 2\|\mathscr A\|\|Q\|\), we deduce 
\[
    |q_j^{(k)}(t)|\le(2\|\mathscr A\|)^k\|Q\|,\quad 0\le t\le T_0,\quad k\ge0.
\]
By Arzel\`a--Ascoli and a diagonal extraction, after passing to a
subsequence we may assume that \(q_j^{(k)}\) converges uniformly on \([0,T_0]\) for every \(k\ge0\). In particular, \(q_j\to q\) uniformly, where \(q\in C^\infty([0,T_0])\), \(q\ge0\), and
\[
    |q^{(k)}(t)|\le(2\|\mathscr A\|)^k\|Q\|,\quad 0\le t\le T_0,\quad k\ge0.
\]
This derivative bound implies that \(q\) is real analytic on
\([0,T_0]\). Since
\[
    \int_0^T q(t)\,\d t=\lim_{j\to\infty}\int_0^T q_j(t)\,\d t=0
\]
and \(q\ge0\), we have \(q=0\) on \([0,T]\). By analyticity we obtain \(q\equiv 0\) on \([0,T_0]\). Therefore
\[
    \int_0^{T_0}q_j(t)\,\d t\to\int_0^{T_0}q(t)\,\d t=0.
\]
This contradicts the assumed observability on \([0,T_0]\), because
\[
    \int_0^{T_0}q_j(t)\,\d t=\int_0^{T_0}\|\mathscr CU(t)x_j\|_Y^2\,\d t\ge c_0\|x_j\|_X^2=c_0.
\]
The lemma follows.
\end{proof}

The following proposition shows that observability at some time and at any time are equivalent. This fact uses only that \(H\) is a bounded self-adjoint operator.
\begin{proposition}
\label{prop:some-any-asymptotic-c}
Let \(G=\mathbb Z\), let \(E\subset\mathbb Z\), and assume that
\(V(n)\to c\in\mathbb R\) as \(|n|\to\infty\). Then the following are
equivalent:
\begin{enumerate}
    \item[(i)] \(E\) is observable at some time for \eqref{eq:DS};
    \item[(ii)] \(E\) is observable at any time for \eqref{eq:DS}.
\end{enumerate}
\end{proposition}

\begin{proof}
The implication \((ii)\Rightarrow(i)\) is immediate.

Assume that \(E\) is observable at some time. Then there exist
\(T_0>0\) and \(c_0>0\) such that
\[
    \int_0^{T_0}
    \|\1_E e^{-itH}f\|_{\ell^2(\mathbb Z)}^2\,\d t
    \ge
    c_0\|f\|_{\ell^2(\mathbb Z)}^2,
    \qquad \forall f\in\ell^2(\mathbb Z).
\]
Since \(V(n)\to c\) as \(|n|\to\infty\), the operator \(H\) is bounded
self-adjoint on \(\ell^2(\mathbb Z)\).  Applying
Lemma~\ref{lem:some-time-any-time-bounded-asymptotic-c} with \eqref{eq:setting},
we obtain that, for every \(T>0\), there exists \(c_T>0\) such that
\[
    \int_0^T
    \|\1_E e^{-itH}f\|_{\ell^2(\mathbb Z)}^2\,\d t
    \ge
    c_T\|f\|_{\ell^2(\mathbb Z)}^2,
    \qquad \forall f\in\ell^2(\mathbb Z).
\]
This is exactly observability at any time for \eqref{eq:DS}.
\end{proof}

\begin{proof}[Proof of Theorem~\ref{thm:1dZ}]
Theorem~\ref{thm:1dZ} follows from Proposition~\ref{prop:some-time-LAC-asymptotic-c}
and Proposition~\ref{prop:some-any-asymptotic-c}.
\end{proof}

\subsection{Heat observability with asymptotically constant potentials on \texorpdfstring{$\mathbb Z$}{Z}}
\label{sec:heat-asymptotically-constant-potentials-Z}

In this subsection we prove the heat analogue of
Theorem~\ref{thm:1dZ}. We keep the notation from
Subsection~\ref{sec:proof-thm-1dZ}.
We consider the heat equation
\begin{equation}\label{eq:heat-asymptotic-potential-Z}
\begin{cases}
    \partial_t u(t)=-Hu(t), & t>0,\\
    u(0)=u_0\in \ell^2(\mathbb Z).
\end{cases}
\end{equation}

We record an abstract comparison lemma, which shows that Schr\"odinger
observability implies heat observability. It uses only the boundedness of the
self-adjoint generator and is independent of the arithmetic condition.

\begin{lemma}
\label{lem:schr-to-heat-bounded-asymptotic}
Let \(X,Y\) be Hilbert spaces. Let
\(\mathscr A\in\mathcal L(X)\) be self-adjoint and let
\(\mathscr C\in\mathcal L(X,Y)\). Suppose that there exist
\(T_0>0\) and \(c_0>0\) such that
\begin{equation}\label{eq:abstract-schrodinger-observability-asymptotic}
    \int_0^{T_0}
    \|\mathscr C e^{-it\mathscr A}x\|_Y^2\,\d t
    \ge
    c_0\|x\|_X^2,
    \qquad
    \forall x\in X.
\end{equation}
Then for every \(T>0\) there exists \(c_T>0\) such that
\begin{equation}\label{eq:abstract-heat-observability-asymptotic}
    \int_0^T
    \|\mathscr C e^{-t\mathscr A}x\|_Y^2\,\d t
    \ge
    c_T\|x\|_X^2,
    \qquad
    \forall x\in X.
\end{equation}
\end{lemma}

\begin{proof}
Suppose, by contradiction, that
\eqref{eq:abstract-heat-observability-asymptotic} fails for some
\(T>0\). Then there exists a sequence \(x_j\in X\) such that $ \|x_j\|_X=1$ and
\begin{equation}\label{eq:heat-small-abstract-asymptotic}
    \int_0^T
    \|\mathscr C e^{-t\mathscr A}x_j\|_Y^2\,\d t
    \to0.
\end{equation}
Set
\[
    F_j(t):=\mathscr C e^{-t\mathscr A}x_j,
    \qquad 0\le t\le T.
\]
Then
\[
    F_j^{(k)}(t)=(-1)^k\mathscr C\mathscr A^k e^{-t\mathscr A}x_j,
    \qquad k\ge0.
\]
For every fixed \(k\ge0\),
\[
    \sup_j
    \|F_j^{(k)}\|_{L^2(0,T;Y)}
    \le
    T^{1/2}
    \|\mathscr C\|_{\mathcal L(X,Y)}
    \|\mathscr A\|_{\mathcal L(X)}^k
    e^{T\|\mathscr A\|_{\mathcal L(X)}}.
\]
Moreover, by \eqref{eq:heat-small-abstract-asymptotic},
\[
    \|F_j\|_{L^2(0,T;Y)}\to0.
\]
Using the Hilbert-valued Gagliardo--Nirenberg interpolation inequality,
we obtain, for every fixed \(k\ge0\),
\[
    \|F_j^{(k)}\|_{L^2(0,T;Y)}\to0.
\]
Indeed, one applies the interpolation inequality to \(F_j\) with top
derivative \(F_j^{(k+1)}\), which is uniformly bounded in
\(L^2(0,T;Y)\).

We next pass from \(L^2\)-smallness to the value at \(t=0\). For an
absolutely continuous \(Y\)-valued function \(g\) on \([0,T]\), the trace
estimate
\[
    \|g(0)\|_Y^2
    \le
    \frac{2}{T}\|g\|_{L^2(0,T;Y)}^2
    +
    2\|g\|_{L^2(0,T;Y)}
      \|g'\|_{L^2(0,T;Y)}
\]
holds. Applying this estimate to \(g=F_j^{(k)}\), and using the uniform
\(L^2\)-boundedness of \(F_j^{(k+1)}\), we get
\[
    F_j^{(k)}(0)\to0.
\]
Equivalently,
\begin{equation}\label{eq:CAk-small-asymptotic}
    \mathscr C\mathscr A^k x_j\to0,
    \qquad
    \forall k\ge0.
\end{equation}

We now prove that
\[
    \mathscr C e^{-it\mathscr A}x_j\to0
\]
uniformly for \(t\in[0,T_0]\). Fix \(N\in\mathbb N\). By Taylor's formula,
\[
    \mathscr C e^{-it\mathscr A}x_j
    =
    \sum_{k=0}^{N-1}
    \frac{(-it)^k}{k!}\mathscr C\mathscr A^k x_j
    +
    \mathscr C\left(
    e^{-it\mathscr A}
    -
    \sum_{k=0}^{N-1}\frac{(-it\mathscr A)^k}{k!}
    \right)x_j .
\]
Since \(\mathscr A\) is self-adjoint, the scalar Taylor remainder estimate
applied through the spectral theorem gives
\[
    \left\|
    e^{-it\mathscr A}
    -
    \sum_{k=0}^{N-1}\frac{(-it\mathscr A)^k}{k!}
    \right\|_{\mathcal L(X)}
    \le
    \frac{\bigl(T_0\|\mathscr A\|_{\mathcal L(X)}\bigr)^N}{N!},
    \qquad
    0\le t\le T_0.
\]
Therefore,
\[
    \sup_{0\le t\le T_0}
    \left\|
    \mathscr C\left(
    e^{-it\mathscr A}
    -
    \sum_{k=0}^{N-1}\frac{(-it\mathscr A)^k}{k!}
    \right)x_j
    \right\|_Y
    \le
    \|\mathscr C\|_{\mathcal L(X,Y)}
    \frac{\bigl(T_0\|\mathscr A\|_{\mathcal L(X)}\bigr)^N}{N!}.
\]
For fixed \(N\), the finite Taylor sum tends to \(0\) uniformly in
\(t\in[0,T_0]\) by \eqref{eq:CAk-small-asymptotic}. Letting
\(N\to\infty\), we get
\[
    \sup_{0\le t\le T_0}
    \|\mathscr C e^{-it\mathscr A}x_j\|_Y
    \to0.
\]
Consequently,
\[
    \int_0^{T_0}
    \|\mathscr C e^{-it\mathscr A}x_j\|_Y^2\,\d t
    \to0,
\]
which contradicts
\eqref{eq:abstract-schrodinger-observability-asymptotic}, because
\(\|x_j\|_X=1\). The lemma follows.
\end{proof}

We now prove the arithmetic characterization for the heat equation.

\begin{proof}[Proof of Theorem \ref{thm:heat-local-arithmetic-Z}]
We first prove \((iii)\Rightarrow(ii)\). Assume that \(E\) satisfies LAC.
By Proposition~\ref{prop:some-time-LAC-asymptotic-c}, \(E\) is observable
at some time for the Schr\"odinger equation generated by \(H\). Hence
there exist \(T_0>0\) and \(c_0>0\) such that
\[
    \int_0^{T_0}
    \|\1_E e^{-itH}f\|_{\ell^2(\mathbb Z)}^2\,\d t
    \ge
    c_0\|f\|_{\ell^2(\mathbb Z)}^2,
    \qquad
    \forall f\in \ell^2(\mathbb Z).
\]
Since \(V(n)\to c\) as \(|n|\to\infty\), the operator \(H\) is bounded
self-adjoint on \(\ell^2(\mathbb Z)\). 
Applying Lemma~\ref{lem:schr-to-heat-bounded-asymptotic} with \begin{align}
 X=Y=\ell^2(\Z),\quad\mathscr A=H\quad{\rm and}\quad\mathscr C=\1_E\cdot{\rm Id}_{\ell^2(\Z)},
\end{align}
we derive that, for every \(T>0\), there exists \(c_T>0\) such that
\[
    \int_0^T
    \|\1_E e^{-tH}f\|_{\ell^2(\mathbb Z)}^2\,\d t
    \ge
    c_T\|f\|_{\ell^2(\mathbb Z)}^2,
    \qquad
    \forall f\in \ell^2(\mathbb Z).
\]
This is exactly heat observability at any time for
\eqref{eq:heat-asymptotic-potential-Z}. Therefore $
    (iii)\Rightarrow(ii).
$

The implication \((ii)\Rightarrow(i)\) is immediate.

It remains to prove \((i)\Rightarrow(iii)\). We argue by contraposition.
Assume that \(E\) does not satisfy LAC. By
Lemma~\ref{lem:arithmetic-quasimodes-H0}, there exist
\(f_N\in\ell^2(\mathbb Z)\) and \(\mu_N\in[c,c+4]\) such that
\[
    \|f_N\|_{\ell^2(\mathbb Z)}=1,
    \qquad
    f_N|_E=0,
\]
and
\[
    \|(H_0-\mu_N)f_N\|_{\ell^2(\mathbb Z)}\to0,
    \qquad
    \|f_N\|_{\ell^\infty(\mathbb Z)}\to0 .
\]
We claim that
\[
    \|Wf_N\|_{\ell^2(\mathbb Z)}\to0.
\]
Indeed, fix \(\varepsilon>0\). Since \(W(n)\to0\) as \(|n|\to\infty\),
choose \(R>0\) such that
\[
    \sup_{|n|>R}|W(n)|\le\varepsilon.
\]
Then
\[
\begin{aligned}
    \|Wf_N\|_{\ell^2(\mathbb Z)}^2
    &\le
    \|W\|_{\ell^\infty(\mathbb Z)}^2
    \sum_{|n|\le R}|f_N(n)|^2
    +
    \varepsilon^2
    \sum_{|n|>R}|f_N(n)|^2                                      \\
    &\le
    \|W\|_{\ell^\infty(\mathbb Z)}^2(2R+1)
    \|f_N\|_{\ell^\infty(\mathbb Z)}^2
    +
    \varepsilon^2 .
\end{aligned}
\]
Letting \(N\to\infty\), and then \(\varepsilon\to0\), gives
\[
    \|Wf_N\|_{\ell^2(\mathbb Z)}\to0.
\]
Therefore
\begin{equation}\label{eq:heat-asymptotic-Weyl-H}
    \|(H-\mu_N)f_N\|_{\ell^2(\mathbb Z)}
    \le
    \|(H_0-\mu_N)f_N\|_{\ell^2(\mathbb Z)}
    +
    \|Wf_N\|_{\ell^2(\mathbb Z)}
    \to0.
\end{equation}

Let \(T>0\) be arbitrary. Since \(H\) is bounded, Duhamel's formula gives
\[
    e^{-tH}f_N-e^{-t\mu_N}f_N
    =
    -\int_0^t
    e^{-(t-s)H}e^{-s\mu_N}
    (H-\mu_N)f_N\,\d s .
\]
Set
\[
    M_c:=\max\{|c|,|c+4|\}.
\]
Since \(\mu_N\in[c,c+4]\), we have
\[
    |e^{-s\mu_N}|\le e^{T M_c},
    \qquad 0\le s\le T.
\]
Moreover,
\[
    \|e^{-(t-s)H}\|_{\mathcal L(\ell^2(\mathbb Z))}
    \le e^{T\|H\|_{\mathcal L(\ell^2(\mathbb Z))}},
    \qquad 0\le s\le t\le T.
\]
Hence
\[
    \|e^{-tH}f_N-e^{-t\mu_N}f_N\|_{\ell^2(\mathbb Z)}
    \le
    t e^{T\left(\|H\|_{\mathcal L(\ell^2(\mathbb Z))}+M_c\right)}
    \|(H-\mu_N)f_N\|_{\ell^2(\mathbb Z)}.
\]
Since \(f_N|_E=0\), we have
\[
    \1_E e^{-t\mu_N}f_N=0.
\]
Therefore
\[
    \|\1_E e^{-tH}f_N\|_{\ell^2(\mathbb Z)}
    \le
    t e^{T\left(\|H\|_{\mathcal L(\ell^2(\mathbb Z))}+M_c\right)}
    \|(H-\mu_N)f_N\|_{\ell^2(\mathbb Z)}.
\]
Integrating in time, we obtain
\[
\begin{aligned}
    \int_0^T
    \|\1_E e^{-tH}f_N\|_{\ell^2(\mathbb Z)}^2\,\d t
    &\le
    e^{2T\left(\|H\|_{\mathcal L(\ell^2(\mathbb Z))}+M_c\right)}
    \|(H-\mu_N)f_N\|_{\ell^2(\mathbb Z)}^2
    \int_0^T t^2\,\d t                                      \\
    &=
    \frac{T^3}{3}
    e^{2T\left(\|H\|_{\mathcal L(\ell^2(\mathbb Z))}+M_c\right)}
    \|(H-\mu_N)f_N\|_{\ell^2(\mathbb Z)}^2
    \to0.
\end{aligned}
\]
On the other hand, $
    \|f_N\|_{\ell^2(\mathbb Z)}=1.
$
Thus the heat observability inequality
\eqref{eq:heat-observability-asymptotic-Z} cannot hold at this time \(T\).
Since \(T>0\) was arbitrary, \(E\) is not heat observable at any time; in
particular, it is not heat observable at some time. This proves the
contrapositive of \((i)\Rightarrow(iii)\).
The theorem is proved.
\end{proof}

\subsection{Observable sets in high dimensions}

We first show that high thickness does not imply observability in higher
dimensions, in contrast to the one-dimensional result that
\(\frac12\)-thickness implies observability.

\begin{proposition}\label{prop:Z2}
Let \(G=\mathbb Z^2\), and let $
    H=-\Delta_{\rm disc}
$. Set
\[
    D:=\{(m,n)\in\mathbb Z^2:m=n\},
    \qquad
    E:=\mathbb Z^2\setminus D .
\]
Then \(E\) is not observable at any time \(T>0\).
\end{proposition}

\begin{proof}
Define
\[
    \psi(m,n):=
    \begin{cases}
        (-1)^m, & m=n,\\
        0, & m\ne n.
    \end{cases}
\]
Then \(\psi|_E=0\). We first claim that
\[
    H\psi=4\psi
\]
pointwise on \(\mathbb Z^2\).

Indeed, if \((m,n)\in D\), then its four nearest neighbours do not belong
to \(D\). Hence
\[
    (H\psi)(m,n)=4\psi(m,n).
\]
If \((m,n)\notin D\), then \(\psi(m,n)=0\). When \(|m-n|>1\), none of
the nearest neighbours of \((m,n)\) belongs to \(D\), and hence
\[
    (H\psi)(m,n)=0=4\psi(m,n).
\]
It remains to consider the case \(|m-n|=1\). If \(m-n=1\), then the two
neighbours of \((m,n)\) lying on \(D\) are $
    (m-1,n),
    (m,n+1),
$
and their values are
\[
    \psi(m-1,n)=(-1)^{m-1},
    \qquad
    \psi(m,n+1)=(-1)^m .
\]
These two values cancel. Similarly, if \(n-m=1\), then the two neighbours
of \((m,n)\) lying on \(D\) are
$
    (m+1,n),
    (m,n-1),
$
and their values are
\[
    \psi(m+1,n)=(-1)^{m+1},
    \qquad
    \psi(m,n-1)=(-1)^m ,
\]
which also cancel. Therefore
\[
    (H\psi)(m,n)=0=4\psi(m,n)
\]
for all \((m,n)\notin D\). This proves \(H\psi=4\psi\).

Although \(\psi\notin\ell^2(\mathbb Z^2)\), its truncations yield
invisible quasimodes. For \(N\in\mathbb N\), define
\[
    \psi_N(m,n):=
    \begin{cases}
        (-1)^m, & m=n,\ |m|\le N,\\
        0, & \text{otherwise}.
    \end{cases}
\]
Then
$
    \spt\psi_N\subset D,
    \psi_N|_E=0,
$
and
$
    \|\psi_N\|_{\ell^2(\mathbb Z^2)}^2=2N+1.
$
Moreover, since \(H\psi=4\psi\), the error $
    (H-4)\psi_N $
is supported only near the two endpoints \((N,N)\) and \((-N,-N)\) of the
truncated diagonal segment. Therefore there exists a constant \(C>0\),
independent of \(N\), such that
\[
    \|(H-4)\psi_N\|_{\ell^2(\mathbb Z^2)}\le C.
\]
Set
\[
    f_N:=\frac{\psi_N}{\|\psi_N\|_{\ell^2(\mathbb Z^2)}}.
\]
Then
\[
    \|f_N\|_{\ell^2(\mathbb Z^2)}=1,
    \qquad
    f_N|_E=0,
\]
and
\[
    \|(H-4)f_N\|_{\ell^2(\mathbb Z^2)}
    =
    \frac{\|(H-4)\psi_N\|_{\ell^2(\mathbb Z^2)}}
    {\|\psi_N\|_{\ell^2(\mathbb Z^2)}}
    \le
    \frac{C}{\sqrt{2N+1}}
    \longrightarrow 0.
\]

Fix \(T>0\). By Duhamel's formula, for every \(t\in[0,T]\),
\[
    e^{-itH}f_N-e^{-4it}f_N
    =
    -i\int_0^t
    e^{-i(t-s)H}e^{-4is}(H-4)f_N\,\d s .
\]
Since \(e^{-itH}\) is unitary on \(\ell^2(\mathbb Z^2)\), it follows that
\[
    \|e^{-itH}f_N-e^{-4it}f_N\|_{\ell^2(\mathbb Z^2)}
    \le
    t\|(H-4)f_N\|_{\ell^2(\mathbb Z^2)}.
\]
Using \(f_N|_E=0\), we have $
    \1_E e^{-4it}f_N=0.
$
Hence
\[
\begin{aligned}
    \|\1_E e^{-itH}f_N\|_{\ell^2(\mathbb Z^2)}
    &\le
    \|e^{-itH}f_N-e^{-4it}f_N\|_{\ell^2(\mathbb Z^2)}  \\
    &\le
    t\|(H-4)f_N\|_{\ell^2(\mathbb Z^2)} .
\end{aligned}
\]
Consequently,
\[
\begin{aligned}
    \int_0^T
    \|\1_E e^{-itH}f_N\|_{\ell^2(\mathbb Z^2)}^2\,\d t
    &\le
    \int_0^T
    t^2\|(H-4)f_N\|_{\ell^2(\mathbb Z^2)}^2\,\d t  \\
    &=
    \frac{T^3}{3}
    \|(H-4)f_N\|_{\ell^2(\mathbb Z^2)}^2
    \longrightarrow 0.
\end{aligned}
\]

If \(E\) were observable at time \(T\), then there would exist a constant
\(C_T>0\) such that
\[
    \|u_0\|_{\ell^2(\mathbb Z^2)}^2
    \le
    C_T
    \int_0^T
    \|\1_E e^{-itH}u_0\|_{\ell^2(\mathbb Z^2)}^2\,\d t
\]
for all \(u_0\in\ell^2(\mathbb Z^2)\). Taking \(u_0=f_N\), we obtain
\[
    1
    =
    \|f_N\|_{\ell^2(\mathbb Z^2)}^2
    \le
    C_T
    \int_0^T
    \|\1_E e^{-itH}f_N\|_{\ell^2(\mathbb Z^2)}^2\,\d t
    \longrightarrow 0,
\]
which is impossible. Therefore \(E\) is not observable at time \(T\).
Since \(T>0\) was arbitrary, \(E\) is not observable at any time.
\end{proof}

We now prove exterior observability on $\Z^d$. 
\begin{proof}[Proof of Theorem \ref{thm:Z^d}]
Since \(V\in\ell^\infty(\mathbb Z^d;\mathbb R)\), the operator
\(H=-\Delta_{\rm disc}+V\) is bounded and self-adjoint on
\(\ell^2(\mathbb Z^d)\). Set \(F:=E^c\). Let
\[
    U(t):=e^{-itH},
    \qquad
    P_F:=\1_F\cdot{\rm Id}_{\ell^2(\mathbb Z^d)} .
\]
Then \(P_F\) is an orthogonal projection of finite rank. For \(T>0\),
define
\[
    K_T
    :=
    \int_0^T U(t)^*P_FU(t)\,dt .
\]
The integral is understood in the operator norm. Since \(H\) is bounded,
\(t\mapsto U(t)^*P_FU(t)\) is norm-continuous, and since each
\(U(t)^*P_FU(t)\) has finite rank, \(K_T\) is compact. Moreover,
\(K_T\) is positive and $
    0\le K_T\le T\cdot{\rm Id}_{\ell^2(\Z^{d})}.$

For every \(f\in\ell^2(\mathbb Z^d)\), using unitarity of \(U(t)\), we
have
\[
\begin{aligned}
    \int_0^T
    \|\1_E U(t)f\|_{\ell^2(\mathbb Z^d)}^2\,dt
    &=
    \int_0^T
    \bigl(
        \|U(t)f\|_{\ell^2(\mathbb Z^d)}^2
        -
        \|P_FU(t)f\|_{\ell^2(\mathbb Z^d)}^2
    \bigr)\,dt                                      \\
    &=
    T\|f\|_{\ell^2(\mathbb Z^d)}^2
    -
    \langle K_Tf,f\rangle .
\end{aligned}
\]
Thus it is enough to prove that $
    \|K_T\|_{\mathcal L(\ell^2)}<T$.

Suppose, by contradiction, that $
    \|K_T\|_{\mathcal L(\ell^2)}=T.$
If \(F=\emptyset\), the conclusion is immediate. Hence we assume \(F\ne\emptyset\). Since \(K_T\) is compact, positive and nonzero, there exists
\(f\in\ell^2(\mathbb Z^d)\) with $
    \|f\|_{\ell^2}=1$ 
such that $
    K_Tf=Tf.$
Hence
\[
    \int_0^T
    \|P_FU(t)f\|_{\ell^2(\mathbb Z^d)}^2\,dt
    =
    T.
\]
But
\[
    \|P_FU(t)f\|_{\ell^2}
    \le
    \|U(t)f\|_{\ell^2}
    =
    \|f\|_{\ell^2}
    =
    1.
\]
Therefore
\[
    \|P_FU(t)f\|_{\ell^2}=1,
    \qquad
    0\le t\le T.
\]
Since \(P_F\) is an orthogonal projection, this implies
\[
    U(t)f\in {\rm Ran}(P_F),
    \qquad
    0\le t\le T.
\]
In particular, $
    f\in{\rm Ran}(P_F).$

Because \(H\) is bounded, the map \(t\mapsto U(t)f\) is analytic in
\(\ell^2(\mathbb Z^d)\). Since \(U(t)f\in{\rm Ran}(P_F)\) for
\(0\le t\le T\), differentiating at \(t=0\) gives
\[
    H^k f\in{\rm Ran}(P_F),
    \qquad
    \forall k\ge0.
\]
Set
\[
    \mathcal M:=\operatorname{span}\{H^k f:k\ge0\}.
\]
Then \(\mathcal M\) is a nonzero finite-dimensional subspace of
\({\rm Ran}(P_F)\), and $
    H\mathcal M\subset \mathcal M.$
Since \(H\) is self-adjoint, the restriction \(H|_{\mathcal M}\) has an
eigenvector. Hence there exist
\[
    0\ne \psi\in\mathcal M\subset{\rm Ran}(P_F),
    \qquad
    \lambda\in\mathbb R,
\]
such that $
    H\psi=\lambda\psi.$
Thus \(\psi\) is a nonzero finitely supported eigenfunction of \(H\),
supported inside the finite set \(F\).

We now show that this is impossible. Let $
    S:=\operatorname{spt}\psi.$
Since \(S\) is finite and nonempty, choose \(n_0\in S\) whose first
coordinate is maximal among all points of \(S\). Set $
    y:=n_0+e_1.$
Then \(y\notin S\). Moreover, the only neighbor of \(y\) which can belong
to \(S\) is \(n_0\). Indeed, any other neighbor of \(y\) has first
coordinate strictly larger than the first coordinate of \(n_0\), and
therefore cannot belong to \(S\).

Since \(\psi(y)=0\), evaluating \(H\psi=\lambda\psi\) at \(y\) gives
\[
    0
    =
    (H\psi)(y)
    =
    -\sum_{z\sim y}\psi(z)
    +(2d+V(y))\psi(y)
    =
    -\psi(n_0).
\]
Hence \(\psi(n_0)=0\), contradicting \(n_0\in S\). Therefore no nonzero
finitely supported eigenfunction of \(H\) exists.

This contradiction proves
\[
    \|K_T\|_{\mathcal L(\ell^2)}<T.
\]
Consequently,
\[
\begin{aligned}
    \int_0^T
    \|\1_EU(t)f\|_{\ell^2(\mathbb Z^d)}^2\,dt
    &=
    T\|f\|_{\ell^2(\mathbb Z^d)}^2
    -
    \langle K_Tf,f\rangle                                  \\
    &\ge
    \bigl(T-\|K_T\|_{\mathcal L(\ell^2)}\bigr)
    \|f\|_{\ell^2(\mathbb Z^d)}^2 .
\end{aligned}
\]
Thus the observability inequality holds with
\[
    C_T
    =
    \bigl(T-\|K_T\|_{\mathcal L(\ell^2)}\bigr)^{-1}.
\]
The theorem is proved.
\end{proof}

\section{Observable sets on discrete tori}\label{sec:3}

\subsection{An equivalent criterion for observability on finite graphs}\label{sec:3.1}

In this subsection, we study the free Schr\"odinger equation \eqref{eq:DS} on a finite graph $G=(\V,\E)$. In this setting, the discrete Laplacian $\Delta_{\mathrm{disc}}$ is a finite-dimensional self-adjoint operator, and its spectral decomposition is standard. The following standard theorem gives an equivalent condition between observability and zero sets of eigenfunctions.

\begin{theorem}\label{thm:finite}
Let \(G=(\mathcal{V},\mathcal{E})\) be a finite graph with discrete Laplacian
\(\Delta_{\mathrm{disc}}\), and let \(E\subseteq \mathcal{V}\). Consider the
free Schr\"odinger equation
\[
    i\partial_t u=-\Delta_{\mathrm{disc}}u .
\]
Then the following are equivalent:
\begin{enumerate}
  \item[(i)] \(E\) is observable at some time \(T>0\);
  \item[(ii)] \(E\) is observable at any time \(T>0\);
  \item[(iii)] for every \(\lambda\in\mathbb R\),
  \[
  \ker(\Delta_{\mathrm{disc}}-\lambda)\cap
  \{f\in\ell^2(\mathcal V): f|_E\equiv0\}=\{0\}.
  \]
\end{enumerate}
\end{theorem}

\begin{proof}
It is immediate that \((ii)\Rightarrow(i)\). We first prove
\((iii)\Rightarrow(ii)\).

Let \(\{\psi_k\}_{k=1}^{n}\) be an orthonormal eigenbasis of
\(\Delta_{\mathrm{disc}}\) in \(\ell^2(\mathcal V)\), with real eigenvalues
\(\{\lambda_k\}_{k=1}^{n}\), where \(n=|\mathcal V|\). Expanding the initial
datum as
\[
  u_0=\sum_{k=1}^{n} c_k \psi_k,
  \qquad
  \bm{c}=(c_1,\dots,c_n)^\top,
  \qquad
  \|u_0\|_{\ell^2(\mathcal V)}^2=\bm{c}^*\bm{c},
\]
the solution is
\[
  u(t)=\sum_{k=1}^{n} c_k e^{i t\lambda_k}\psi_k .
\]
For fixed \(T>0\), define the Gramian matrix \(\G=\G(T,E)\) by
\[
  \G_{jk}
  =
  \big\langle \psi_k|_E,\psi_j|_E\big\rangle_{\ell^2(E)}
  \int_0^T e^{i t(\lambda_k-\lambda_j)}\,\d t .
\]
Then
\[
  \int_0^T \|u(t)\|_{\ell^2(E)}^2\,\d t
  =
  \bm{c}^*\G\bm{c}.
\]
We claim that \(\G\) is positive definite. Suppose that
\(\bm{c}^*\G\bm{c}=0\). Then
\[
  \int_0^T \|u(t)\|_{\ell^2(E)}^2\,\d t=0.
\]
Since \(t\mapsto \|u(t)\|_{\ell^2(E)}^2\) is nonnegative and continuous, it
follows that
\[
  \|u(t)\|_{\ell^2(E)}=0,
  \qquad 0\le t\le T.
\]
Equivalently,
\[
  \sum_{k=1}^{n} c_k e^{i t\lambda_k}\psi_k(x)=0,
  \qquad x\in E,\quad 0\le t\le T.
\]
Let \(\{\Lambda_m\}\) be the distinct eigenvalues among
\(\{\lambda_k\}_{k=1}^n\), and group the above sum according to eigenspaces:
\[
  \sum_m e^{i t\Lambda_m} w_m(x)=0,
  \qquad
  w_m:=\sum_{k:\lambda_k=\Lambda_m} c_k\psi_k,
  \qquad x\in E .
\]
The exponentials \(\{e^{i t\Lambda_m}\}_m\), with distinct real frequencies,
are linearly independent on every interval of positive length. Hence
\[
  w_m(x)=0,
  \qquad x\in E,
\]
for every \(m\). Thus \(w_m|_E\equiv0\). Since each \(w_m\) belongs to the
eigenspace of \(\Delta_{\mathrm{disc}}\) associated with \(\Lambda_m\),
assumption (iii) implies that \(w_m=0\) for every \(m\). Therefore
\[
  u_0=\sum_m w_m=0,
\]
and hence \(\bm{c}=0\). This proves that \(\G\) is positive definite.

Let \(\mu_{\min}(T,E)>0\) be the smallest eigenvalue of \(\G\). Then
\[
  \int_0^T \|u(t)\|_{\ell^2(E)}^2\,\d t
  =
  \bm{c}^*\G\bm{c}
  \ge
  \mu_{\min}(T,E)\,\bm{c}^*\bm{c}
  =
  \mu_{\min}(T,E)\,\|u_0\|_{\ell^2(\mathcal V)}^2.
\]
Therefore
\[
  \|u_0\|_{\ell^2(\mathcal V)}^2
  \le
  \frac{1}{\mu_{\min}(T,E)}
  \int_0^T \|u(t)\|_{\ell^2(E)}^2\,\d t .
\]
Since \(T>0\) was arbitrary, \(E\) is observable at any time. Thus
\((iii)\Rightarrow(ii)\).

It remains to prove \((i)\Rightarrow(iii)\). We argue by contradiction.
Suppose that there exist \(\lambda_0\in\mathbb R\) and a nonzero
\(\psi_0\in\ker(\Delta_{\mathrm{disc}}-\lambda_0)\) such that
\[
    \psi_0|_E\equiv0.
\]
Taking \(u_0=\psi_0\), the corresponding solution is
\[
    u(t)=e^{i t\lambda_0}\psi_0 .
\]
Hence
\[
  \|u(t)\|_{\ell^2(E)}^2
  =
  \sum_{x\in E}\big|e^{i t\lambda_0}\psi_0(x)\big|^2
  =
  \sum_{x\in E}|\psi_0(x)|^2
  =
  0
\]
for every \(t\ge0\). Therefore, for every \(T>0\),
\[
  \int_0^T \|u(t)\|_{\ell^2(E)}^2\,\d t=0,
\]
whereas
\[
  \|u_0\|_{\ell^2(\mathcal V)}^2
  =
  \|\psi_0\|_{\ell^2(\mathcal V)}^2
  >
  0.
\]
Thus no observability inequality can hold at any time \(T>0\), contradicting
(i). Hence (iii) follows.

Combining the three implications gives the equivalence of (i), (ii), and
(iii).
\end{proof}

\subsection{A positive density construction  for unobservable sets on discrete tori}\label{sec:3.2}

On discrete tori, the combinatorial Laplacian $\Delta_{\mathrm{disc}}$ acts on $f:(\Z/N\Z)^d\to\mathbb{C}$ by
\[
  (\Delta_{\mathrm{disc}} f)(x) \;=\; \sum_{j=1}^{d}\bigl(f(x+e_j)+f(x-e_j)\bigr) - 2d\,f(x).
\]
For $k:=(k_1,\dots,k_d)\in(\Z/N\Z)^d$, define the character
\[
  \varphi_k(x) \;=\; \exp\!\Bigl(\frac{2\pi i}{N}\,\ip{k}{x}\Bigr),\qquad x\in(\Z/N\Z)^d,
\]
where $\ip{k}{x}=k_1x_1+\cdots+k_dx_d$ (interpreted modulo $N$). A direct computation gives that
\[
\Delta_{\mathrm{disc}}\varphi_k=\left(2\sum_{j=1}^{d}\cos\!\Bigl(\frac{2\pi k_j}{N}\Bigr)-2d \right)\varphi_k:=\mu(k)\varphi_k.
\]

The imaginary parts $\psi_k^{\sin}(x)=\sin(\tfrac{2\pi}{N}\ip{k}{x})$, when nonzero, are real-valued eigenfunctions with eigenvalue $\mu(k)$. By the standard structure theory for finite abelian groups, the characters $\{\varphi_k\}_{k\in(\mathbb Z/N\mathbb Z)^d}$ form an orthogonal basis of the Hilbert space $L^2\!\bigl((\mathbb Z/N\mathbb Z)^d\bigr)$.

Fix $k\in(\mathbb Z/N\mathbb Z)^d$. Consider the group homomorphism
\[
  F_k:(\mathbb Z/N\mathbb Z)^d\longrightarrow \mathbb Z/N\mathbb Z,\qquad
  F_k(x)=\ip{k}{x}\ (\mathrm{mod}\ N),
\]
and write
\[
  d_0=\gcd(k_1,\dots,k_d,N)\in\Z_{\ge 1}.
\]

On the additive group $\Z/N\Z$, we introduce
\[
  S_{\sin} \;:=\; \bigl\{\, r\in \Z/N\Z:\ \sin\!\bigl(\frac{2\pi r}{N}\bigr)=0 \,\bigr\}
  \;=\;
  \begin{cases}
    \{\,0\,\}, & 2\nmid N,\\
    \{\,0,\ N/2\,\}, & 2\mid N,
  \end{cases}
\]
Here $N/2$ denotes the unique nonzero residue in $\Z/N\Z$ satisfying $2\cdot (N/2)\equiv 0$.

We have the following proposition for counting the size of the zero set $Z(\psi_k^{\sin})$.

\begin{proposition}[Zero set size via group homomorphisms]\label{prop:zero-count}
With the notation above, we have
\[
  Z\bigl(\psi_k^{\sin}\bigr)
  \;=\; F_k^{-1}\!\bigl(S_{\sin}\cap \mathrm{Im}\,F_k\bigr),\qquad
  \bigl|Z(\psi_k^{\sin})\bigr|
  \;=\; \bigl|S_{\sin}\cap \mathrm{Im}\,F_k\bigr|\cdot d_0\,N^{d-1},
\]
\end{proposition}

\begin{proof}
Consider the subgroup $H\subset \Z/N\Z$ generated by the residues $[k_1],\dots,[k_d]$. The preimage of $H$ under the natural projection $\Z\to\Z/N\Z$ is the ideal $\langle k_1,\dots,k_d,N\rangle=d_0\Z$, where $d_0=\gcd(k_1,\dots,k_d,N)$. Hence $H=\langle [d_0]\rangle$ and therefore $|\mathrm{Im}\,F_k|=|H|=N/d_0$. Since $F_k$ is a homomorphism between finite groups, all fibers over residues in $\mathrm{Im}\,F_k$ have the same cardinality; explicitly,
$$
  \bigl|F_k^{-1}(r)\bigr| \;=\; |\ker F_k|
  \;=\; \frac{|(\mathbb Z/N\mathbb Z)^d|}{|\mathrm{Im}\,F_k|}
  \;=\; d_0\,N^{d-1}
  \qquad \text{for every } r\in \mathrm{Im}\,F_k.
$$

Next, by the elementary characterizations $\sin\theta=0$ iff $\theta\in \pi\Z$, a residue $r\in\Z/N\Z$ satisfies $\sin(2\pi r/N)=0$ exactly when $r\equiv 0$ (and also $r\equiv N/2$ if $2\mid N$). These are precisely the set $S_{\sin}$ appearing in the statement.

Finally, since $\psi_k^{\sin}(x)=\sin\!\bigl(\tfrac{2\pi}{N}F_k(x)\bigr)$, an element $x\in(\mathbb Z/N\mathbb Z)^d$ is a zero of $\psi_k^{\sin}$ if and only if $F_k(x)$ lies in $S_{\sin}$. Thus
$$
  Z\bigl(\psi_k^{\sin}\bigr) \;=\; F_k^{-1}\!\bigl(S_{\sin}\cap \mathrm{Im}\,F_k\bigr).
$$
Because the preimage of a subset is the disjoint union of the fibers over its elements, and each fiber has size $d_0\,N^{d-1}$, the cardinalities are
$$
  \bigl|Z(\psi_k^{\sin})\bigr|
  \;=\; \bigl|S_{\sin}\cap \mathrm{Im}\,F_k\bigr|\cdot d_0\,N^{d-1},
$$
as claimed.
\end{proof}

Now we begin to construct high-density unobservable sets via product eigenfunctions. For $r=(r_1,\dots,r_d)\in (\mathbb Z/N\mathbb Z)^d$, we consider
\begin{equation}\label{eq:Psi}
  \Psi_r(x)=\prod_{j=1}^d \sin\!\Bigl(\frac{2\pi r_j x_j}{N}\Bigr),\qquad x=(x_1,\dots,x_d)\in(\mathbb Z/N\mathbb Z)^d.  
\end{equation}
Expanding the product shows that
\[
  \Psi_r(x) \;=\; \sum_{s\in\{\pm 1\}^d} c_s \,\exp\!\Bigl(\frac{2\pi i}{N}\,\ip{s\circ r}{x}\Bigr),
\]
where $(s\circ r)_j=s_j r_j$ and $c_s\neq 0$. Since
$\mu(s\circ r)=2\sum_{j=1}^d \cos(\tfrac{2\pi r_j}{N})-2d$ is independent of $s$, $\Psi_r$ is a Laplcian eigenfunction with eigenvalue $\mu(r)$.

Set $Z_j(r):=\{x\in(\mathbb Z/N\mathbb Z)^d:\ \sin(\tfrac{2\pi r_j x_j}{N})=0\}$ and
\[
  E_{\mathrm{prod}}(r):=\bigcup_{j=1}^d Z_j(r).
\]
If \(\Psi_r\not\equiv0\), then \(\Psi_r\) vanishes on \(E_{\mathrm{prod}}(r)\), and hence \(E_{\mathrm{prod}}(r)\) is unobservable by Theorem~\ref{thm:finite}.

\begin{theorem}[Product construction and its size]\label{thm:product-size}
Let $N\ge 3$ and $d\ge 1$. If $4\nmid N$, let $p_{\min}$ denote the smallest odd prime divisor of $N$
(if $N$ is odd prime, then $p_{\min}=N$).
Define $r_j$ coordinate-wise by
\[
  r_j \;=\;
  \begin{cases}
    N/4, & \text{if } 4\mid N,\\
    N/p_{\min}, & \text{if } 4\nmid N,
  \end{cases}
  \qquad j=1,\dots,d,
\]
and set $r=(r_1,\dots,r_d)$ and $E_{\mathrm{prod}}:=E_{\mathrm{prod}}(r)$.

Then $E_{\mathrm{prod}}$ is not observable at any time and its cardinality is
\[
  |E_{\mathrm{prod}}|
  \;=\;
  \begin{cases}
    N^d-(N/2)^d \;=\; N^d\bigl(1-2^{-d}\bigr), & \text{if } 4\mid N,\\[0.4em]
    N^d-\bigl(N-N/p_{\min}\bigr)^d
    \;=\; N^d\Bigl[1-\bigl(1-\tfrac{1}{p_{\min}}\bigr)^d\Bigr], & \text{if } 4\nmid N.
  \end{cases}
\]
\end{theorem}

\begin{proof}
We already showed $\Psi_r$ defined in \eqref{eq:Psi} is an eigenfunction and $\Psi_r|_{E_{\mathrm{prod}}}\equiv 0$.
It remains to check that $\Psi_r\not\equiv0$ and to count.

If $4\mid N$ and $r_j=N/4$, then
\[
    \Psi_r(1,\dots,1)=\prod_{j=1}^d \sin(\pi/2)=1,
\]
so $\Psi_r\not\equiv0$. Moreover, if
\[
    A_j(r):=\{a\in\mathbb Z/N\mathbb Z:\ \sin(\tfrac{2\pi r_j a}{N})=0\},
\]
then $A_j(r)$ is the set of even residues and hence $|A_j(r)|=N/2$. Since
\[
    Z_j(r)=(\mathbb Z/N\mathbb Z)^{j-1}\times A_j(r)\times
    (\mathbb Z/N\mathbb Z)^{d-j},
\]
we have $|Z_j(r)|=(N/2)N^{d-1}$. Independence across coordinates shows the complement of $E_{\mathrm{prod}}$ consists of points with all $x_j$ odd, of size $(N/2)^d$,
yielding $|E_{\mathrm{prod}}|=N^d-(N/2)^d$.

If $4\nmid N$, take $r_j=N/p_{\min}$. Then
\[
    \Psi_r(1,\dots,1)=\prod_{j=1}^d \sin(2\pi/p_{\min})\ne0,
\]
because $p_{\min}$ is an odd prime. Moreover, with $A_j(r)$ as above, the only sine zero in the image generated by $r_j$ is $0$, and therefore
\[
    A_j(r)=\{a\in\mathbb Z/N\mathbb Z:\ a\equiv0\ (\mathrm{mod}\ p_{\min})\},
    \qquad |A_j(r)|=N/p_{\min}.
\]
Consequently
\[
    Z_j(r)=(\mathbb Z/N\mathbb Z)^{j-1}\times A_j(r)\times
    (\mathbb Z/N\mathbb Z)^{d-j},
\]
and $|Z_j(r)|=(N/p_{\min})N^{d-1}$. Again independence across coordinates gives
$|((\mathbb Z/N\mathbb Z)^d)\setminus E_{\mathrm{prod}}|=\bigl(N-N/p_{\min}\bigr)^d$, hence the stated formula.
\end{proof}

\begin{proof}[Proof of Theorem~\ref{thm:tori}]
For \(\mathrm{(i)}\), by Theorem~\ref{thm:finite}, observability on the
finite graph \(T_N^1\) is equivalent to the absence of a nonzero
eigenfunction vanishing on \(E\). In the gcd condition, \(x-y\) is understood through arbitrary integer representatives of the residues; since \(N/\gcd(N,2)\mid N\), the condition is independent of the choice of representatives. Since the eigenspaces of
\(\Delta_{\rm disc}\) on \(T_N^1\) are
\[
    \operatorname{span}\{\varphi_0\},\qquad
    \operatorname{span}\{\varphi_{N/2}\}\quad (2\mid N),
\]
and
\[
    \operatorname{span}\{\varphi_k,\varphi_{-k}\},
    \qquad 2k\not\equiv0\pmod N,
\]
a nonzero eigenfunction can vanish on the nonempty set \(E\) only in the last case. For the last case, a function
\[
    a\varphi_k+b\varphi_{-k}
\]
vanishes on \(E\) nontrivially if and only if
\[
    \exp\left(\frac{4\pi i kx}{N}\right)
    \text{ is constant on }E.
\]
Indeed, if $a\varphi_k+b\varphi_{-k}$ vanishes on $E$, then $a,b\ne0$, and after multiplying by $\varphi_k$ we get
\[
    a\exp\left(\frac{4\pi i kx}{N}\right)+b=0,
    \qquad x\in E.
\]
The converse is immediate by choosing $a,b$ so that the preceding identity holds. This constancy condition is equivalent to
\[
    \frac{N}{\gcd(N,2k)}\mid x-y,
    \qquad x,y\in E.
\]
As \(k\) varies, the possible nontrivial divisors
\(\frac{N}{\gcd(N,2k)}\) are exactly the nontrivial divisors of
\(N/\gcd(N,2)\). Hence \(E\) is observable if and only if
\[
    \gcd\left\{
        \frac{N}{\gcd(N,2)},\ x-y:\ x,y\in E
    \right\}=1.
\]

For \(\mathrm{(ii)}\), let \(N_m\to\infty\) be an infinite sequence such
that
\[
    N_m\equiv0\pmod4.
\]
The nonzero function
\[
    \psi(x)=\prod_{j=1}^d \sin\left(\frac{\pi x_j}{2}\right)
\]
is a Laplacian eigenfunction on \((\mathbb Z/N_m\mathbb Z)^d\). Let
\[
    E_{N_m}:=\{x\in(\mathbb Z/N_m\mathbb Z)^d:\ \exists j,\ x_j\equiv0\pmod2\}.
\]
Then \(\psi|_{E_{N_m}}\equiv0\). By Theorem~\ref{thm:finite},
\(E_{N_m}\) is not observable for any \(T>0\). Moreover,
\[
    |E_{N_m}|=N_m^d-(N_m/2)^d=N_m^d(1-2^{-d}).
\]
Thus \(\mathrm{(ii)}\) holds with \(c=1-2^{-d}\). This completes the proof.
\end{proof}

We remark that the product construction over primes yields vanishing density. More precisely, assume $N$ is odd prime and take $r_j=N/p_{\min}=1$ for all $j$. Then $
  |E_{\mathrm{prod}}|=N^d-(N-1)^d \sim d\,N^{d-1}$ as \(N\to\infty\),  
that is, $|E_{\mathrm{prod}}|/N^d\to 0$ as $N\to\infty$.

For $r=(r_1,\dots,r_d)\in (\mathbb Z/N\mathbb Z)^d$, we introduce
\[
  \mathcal{V}_r \;=\; \mathrm{span}\,\bigl\{\, \varphi_{s\circ r}:\ s\in\{\pm 1\}^d \,\bigr\}.
\]
Then $\dim \mathcal{V}_r\le 2^d$ (in fact $=2^d$ if all $r_j\not\equiv 0$ and $r_j\not\equiv N/2$ when $2\mid N$). We conclude the paper by showing that, when $4\mid N$ and $r_j=N/4$, the eigenfunction $\Psi_r$ minimizes the number of nonzero entries among all nonzero vectors in $\mathcal{V}_r$; equivalently, it maximizes the zero set $Z(\psi)$ within $\mathcal{V}_r$.

\begin{proposition}[An uncertainty bound and optimality in $\mathcal{V}_r$]\label{prop:Vr-optimal}

Fix $r$ as in Theorem~\ref{thm:product-size}.
Then for any $0\ne \psi\in \mathcal{V}_r$, we have
\[
  |\mathrm{supp}\,\psi| \;\ge\; \frac{N^d}{2^d}.
\]
Moreover, if $4\mid N$ and $r_j=N/4$ for all $j$, then $\Psi_r$ is a maximizer of $|Z(\psi)|$ among all nonzero $\psi\in\mathcal{V}_r$.
\end{proposition}

\begin{proof}
Let $G=(\mathbb Z/N\mathbb Z)^d$ and let $\widehat{G}$ be its dual (identified with $(\mathbb Z/N\mathbb Z)^d$ via characters).
For any nonzero $f:G\to\mathbb{C}$ the Donoho--Stark uncertainty principle (see for example \cite{DS89, M06}) gives
$|\mathrm{supp}\, f|\cdot |\mathrm{supp}\, \widehat{f}|\ge |G|=N^d$.

Every $\psi\in\mathcal{V}_r$ has Fourier support contained in the set
$\{\,s\circ r:\ s\in\{\pm 1\}^d\,\}$, hence $|\mathrm{supp}\,\widehat{\psi}|\le 2^d$.
Applying the uncertainty inequality yields $|\mathrm{supp}\,\psi|\ge N^d/2^d$.

When $4\mid N$ and $r_j=N/4$, we computed in Theorem~\ref{thm:product-size} that
$\Psi_r(x)\ne 0$ iff all coordinates $x_j$ are odd, so $|\mathrm{supp}\,\Psi_r|=(N/2)^d$.
Thus $\Psi_r$ attains the lower bound in $\mathcal{V}_r$, which implies that $\Psi_r$ is one maximizer of $|Z(\psi)|$
within $\mathcal{V}_r$.
\end{proof}

\appendix
\section{A Logvinenko--Sereda proof of the half-thickness criterion}
\label{app:LS-half-thick}

In this appendix we give an alternative proof, in the free case, that
\(\frac12\)-thickness implies observability in every positive time.  That is,  we
consider \(u(t)=e^{it\Delta_{\rm disc}}u_0\), \(u_0\in \ell^2(\Z)\).

We use the Fourier transform on \(\mathbb Z\)
\[
    \widehat f(\theta)=
    \frac1{\sqrt{2\pi}}
    \sum_{n\in\mathbb Z} f(n)e^{in\theta},
    \qquad \theta\in\mathbb T,
\]
where \(\mathbb T=[-\pi,\pi)\).  The Fourier multiplier of
\(\Delta_{\rm disc}\) is
\[
    \varphi(\theta)=2\cos\theta-2,
    \qquad \theta\in\mathbb T.
\]
Thus on the Fourier side we have
\[
    \widehat{\Delta_{\rm disc}f}(\theta)
    =
    \varphi(\theta)\widehat f(\theta).
\]
For convenience, we introduce the lower Beurling density.
\begin{definition}[lower Beurling density]\label{Beurling}
   For a discrete set $E\subset\Z$, its lower Beurling density $d^{-}\left(E\right)$ is given by
\begin{equation*}
    d^{-}\left(E\right):=\liminf_{R\to\infty}\inf_{x\in\R}\frac{\left|E\cap\left[x-R,x+R\right]\right|}{2R}.
\end{equation*}
\end{definition}

We first establish the following density form of the discrete
Logvinenko--Sereda inequality by Beurling's sampling theorem, see \cite[Theorem 1.1]{GRS18}.

\begin{lemma}[Discrete Logvinenko--Sereda inequality]
\label{lem:app-discrete-LS}
Let \(E\subset\mathbb Z\), and let \(I\subset\mathbb T\) be an interval of
length \(a\in(0,2\pi)\). Assume
\[
    d^-(E)>\frac{a}{2\pi}.
\]
Then there exists a constant \(C_{\rm LS}=C_{\rm LS}(E,a)>0\), independent
of the position of \(I\), such that every \(f\in\ell^2(\mathbb Z)\) with
\({\rm spt}\,\widehat f\subset I\) satisfies
\[
    \|f\|_{\ell^2(\mathbb Z)}^2
    \le
    C_{\rm LS}\|f\|_{\ell^2(E)}^2 .
\]
\end{lemma}

\begin{proof}
By modulation, it is enough to treat the case where \(I\) is centered at
the origin.  Put
$
    \sigma:=\frac{a}{2\pi}.
$
For \(f\in\ell^2(\mathbb Z)\) with \({\rm spt}\,\widehat f\subset I\), define
the Paley--Wiener extension
\[
    F(x)
    =
    \frac1{\sqrt{2\pi}}\int_I \widehat f(\theta)e^{-ix\theta}\,\d\theta,
    \qquad x\in\mathbb R .
\]
Then \(F(n)=f(n)\) for \(n\in\mathbb Z\), and the support of the Fourier transform \(\mathcal F f\) on \(\R\) satisfies
\[
    {\rm spt}\,\mathcal F F\subset[-\pi\sigma,\pi\sigma].
\]
Thus \(F\in\mathcal{PW}_{\sigma}\). Since
\[
    d^-(\sigma E)=\frac{d^-(E)}{\sigma}>1,
\]
Beurling's sampling theorem \cite[Theorem 1.1]{GRS18} applied to the rescaled function
\(G(x):=F(x/\sigma)\in\mathcal{PW}_1\) gives
\[
    \|G\|_{L^2(\mathbb R)}^2
    \le
    C
    \sum_{n\in E}|G(\sigma n)|^2
    =
    C
    \sum_{n\in E}|F(n)|^2 .
\]
Because \(\|G\|_{L^2(\mathbb R)}^2=\sigma\|F\|_{L^2(\mathbb R)}^2\), we get
\[
    \|F\|_{L^2(\mathbb R)}^2
    \le
    C_{\rm LS}
    \sum_{n\in E}|F(n)|^2.
\]
By Plancherel,
\[
    \|F\|_{L^2(\mathbb R)}^2
    =
    \|\widehat f\|_{L^2(\mathbb T)}^2
    =
    \|f\|_{\ell^2(\mathbb Z)}^2,
\]
and since \(F(n)=f(n)\), the desired estimate follows. The independence of
the position of \(I\) follows from the initial modulation, which does not
change any \(\ell^2\)-norm.
\end{proof}

Next we record a simple covering property of the multiplier
\(\varphi(\theta)=2\cos\theta-2\).

\begin{lemma}[Uniform one-interval localization of the multiplier]
\label{lem:app-multiplier-cover}
Let \(a\in(\pi,2\pi)\). Then there exists \(\rho_a>0\) such that for every
\(\lambda\in\mathbb R\) one can find an interval \(J_\lambda\subset\mathbb T\)
of length \(a\) satisfying
\[
    |\varphi(\theta)-\lambda|\ge \rho_a,
    \qquad
    \theta\in\mathbb T\setminus J_\lambda .
\]
\end{lemma}

\begin{proof}
First let \(\lambda\in[-4,0]\). The level set
$
    \varphi^{-1}(\lambda)
$
is either a singleton, when \(\lambda=0\) or \(\lambda=-4\), or a pair
\(\{\theta_\lambda,-\theta_\lambda\}\). In all cases it is contained in
some open interval of \(\mathbb T\) of length strictly smaller than \(a\),
because the maximal circular distance between \(\theta_\lambda\) and
\(-\theta_\lambda\) is \(\pi<a\).

Hence, for each \(\lambda\in[-4,0]\), we may choose an interval
\(J_\lambda^0\subset\mathbb T\) of length strictly smaller than \(a\) and a
number \(r_\lambda>0\) such that
\[
    |\varphi(\theta)-\lambda|\ge r_\lambda,
    \qquad
    \theta\in\mathbb T\setminus J_\lambda^0 .
\]
By compactness of \([-4,0]\), finitely many neighborhoods
\((\lambda_j-r_{\lambda_j}/2,\lambda_j+r_{\lambda_j}/2)\) cover
\([-4,0]\). After enlarging the corresponding \(J_{\lambda_j}^0\)'s, if
necessary, to intervals of length exactly \(a\), and setting
\[
    \rho_a
    :=
    \frac12\min_j r_{\lambda_j}>0,
\]
we obtain the desired estimate for all \(\lambda\in[-4,0]\).

If \(\lambda<-4\), we use the interval chosen for the endpoint
\(\lambda=-4\). Since \(\varphi(\theta)\ge -4\), the lower bound outside
that interval remains valid, possibly after decreasing \(\rho_a\).
Similarly, if \(\lambda>0\), we use the interval chosen for \(\lambda=0\).
This proves the lemma for every \(\lambda\in\mathbb R\).
\end{proof}

\begin{proposition}[Half-thickness implies observability]
\label{prop:app-half-thick-some-time}
Let \(E\subset\mathbb Z\) be \(\frac12\)-thick. Then \(E\) is observable at
some time for the free discrete Schr\"odinger equation on \(\mathbb Z\).
\end{proposition}

\begin{proof}
If \(E\subset\mathbb Z\) is \(\frac12\)-thick, then \(d^-(E)>\frac12\).
Choose
$
    a\in\bigl(\pi,\,2\pi d^-(E)\bigr).
$
Let \(\rho_a>0\) be the constant from
Lemma~\ref{lem:app-multiplier-cover}. For \(\lambda\in\mathbb R\), let
\(J_\lambda\subset\mathbb T\) be an interval of length \(a\) such that
\[
    |\varphi(\theta)-\lambda|\ge \rho_a,
    \qquad
    \theta\in\mathbb T\setminus J_\lambda .
\]
Let \(P_\lambda\) be the Fourier projection
\[
    \widehat{P_\lambda f}
    =
    \1_{J_\lambda}\widehat f .
\]
Since \(a<2\pi d^-(E)\), Lemma~\ref{lem:app-discrete-LS} gives
\[
    \|P_\lambda f\|_{\ell^2(\mathbb Z)}^2
    \le
    C_{\rm LS}\|P_\lambda f\|_{\ell^2(E)}^2,
    \qquad f\in\ell^2(\mathbb Z),
\]
with a constant independent of \(\lambda\). Hence
\[
\begin{aligned}
    \|f\|_{\ell^2(\mathbb Z)}^2
    &=
    \|P_\lambda f\|_{\ell^2(\mathbb Z)}^2
    +
    \|(I-P_\lambda)f\|_{\ell^2(\mathbb Z)}^2                                      \\
    &\le
    C_{\rm LS}\|P_\lambda f\|_{\ell^2(E)}^2
    +
    \|(I-P_\lambda)f\|_{\ell^2(\mathbb Z)}^2                                      \\
    &\le
    2C_{\rm LS}\|f\|_{\ell^2(E)}^2
    +
    (1+2C_{\rm LS})\|(I-P_\lambda)f\|_{\ell^2(\mathbb Z)}^2 .
\end{aligned}
\]
On the other hand, by Plancherel and the definition of \(J_\lambda\),
\[
\begin{aligned}
    \|(I-P_\lambda)f\|_{\ell^2(\mathbb Z)}^2
    &=
    \int_{\mathbb T\setminus J_\lambda}|\widehat f(\theta)|^2\,\d\theta              \\
    &\le
    \rho_a^{-2}
    \int_{\mathbb T\setminus J_\lambda}
    |\varphi(\theta)-\lambda|^2|\widehat f(\theta)|^2\,\d\theta                    \\
    &\le
    \rho_a^{-2}
    \|(\Delta_{\rm disc}-\lambda)f\|_{\ell^2(\mathbb Z)}^2 .
\end{aligned}
\]
Therefore, for all \(\lambda\in\mathbb R\) and all
\(f\in\ell^2(\mathbb Z)\),
\[
    \|f\|_{\ell^2(\mathbb Z)}^2
    \le
    M
    \|(\Delta_{\rm disc}-\lambda)f\|_{\ell^2(\mathbb Z)}^2
    +
    m
    \|f\|_{\ell^2(E)}^2,
\]
where
\[
    M:=\frac{1+2C_{\rm LS}}{\rho_a^2},
    \qquad
    m:=2C_{\rm LS}.
\]
This is precisely the observable resolvent estimate for
\[
    \mathscr{A}=\Delta_{\rm disc},\quad \mathscr{C}=\1_E\cdot{\rm Id}_{\ell^2(\mathbb Z)}.
\]
By the Hautus resolvent criterion, Theorem~\ref{thm:M05}, the free
Schr\"odinger equation is observable at some time. By Proposition \ref{prop:some-any-asymptotic-c}, $E$ is observable at any time.
\end{proof}

\section*{Acknowledgement}
 We are grateful to Chenmin Sun for valuable suggestions and conversations. We sincerely thank Professors Jingwei Guo, Wen Huang, Shiping Liu and Shunlin Shen for their discussions. H.Z. is supported by the National Key R \& D Program of China 2023YFA1010200 and the
 National Natural Science Foundation of China No. 12431004. Z.W. is supported by the National Natural Science Foundation of China No. 12341102.

\subsection*{Availability of Data}  No data were used for the research described in the article.

\subsection*{Declarations of Conflict of Interest}
The authors do not have any possible conflicts of
interest.

\end{document}